\documentclass[11pt]{article}
\usepackage{amsmath,amssymb,amsthm,graphicx,epsfig,subfigure,float,url,color,bbold}

\usepackage{pdfsync}
\usepackage[colorlinks=true]{hyperref}

\newcommand{\Id} {\mathrm{Id}}

\newcommand{\Om} {\Omega}

\newcommand{\lb} {\lambda}

\newcommand{\C} {\mathbb{C}}
\newcommand{\R} {\mathbb{R}}

\renewcommand{\geq}{\geqslant}
\renewcommand{\leq}{\leqslant}

\newtheorem{theorem}{Theorem}[section]
\newtheorem{proposition}[theorem]{Proposition}
\newtheorem{corollary}[theorem]{Corollary}
\newtheorem{definition}[theorem]{Definition}

\newtheorem{lemma}[theorem]{Lemma}

\theoremstyle{definition}\newtheorem{remark}{Remark}
\numberwithin{equation}{section}

\title{On the controllability of quantum  transport in an electronic nanostructure}
\author{Florian M\'ehats\footnote{IRMAR, Univ. Rennes 1 and INRIA Team IPSO, Campus de Beaulieu, 35042 Rennes Cedex, France (\texttt{florian.mehats@univ-rennes1.fr}).}
	\and Yannick Privat\footnote{CNRS, Universit\'e Pierre et Marie Curie (Univ. Paris 6), UMR 7598, Laboratoire Jacques-Louis Lions, F-75005, Paris, France ({\tt yannick.privat@upmc.fr}).}
	\and Mario Sigalotti\footnote{INRIA Saclay-\^Ile-de-France, Team GECO, and CMAP, UMR 7641, \'Ecole Polytechnique, Route de Saclay, 91128 Palaiseau Cedex, France (\texttt{mario.sigalotti@inria.fr}).}
        }

\date{}
\begin{document}
\sloppy
\maketitle
\begin{abstract}
We investigate the controllability of quantum electrons trapped in a two-dimensional device, typically a MOS field-effect transistor. The problem is modeled by the Schr\"odinger equation in a bounded domain coupled to the Poisson equation for the electrical potential. The controller acts on the system through the boundary condition on the potential, on a part of the boundary modeling the gate. We prove that, generically with respect to the shape of the domain and boundary conditions on the gate, the device is controllable. 
We also consider control properties of a more realistic nonlinear version of the device, taking into account the self-consistent electrostatic Poisson potential. 
\end{abstract}

\medskip


\section{Introduction and main results}

In order to comply with the growing needs of ultra-fast, low-consumption and high-functionality operation, microelectronics industry has driven transistor sizes to the nanometer scale \cite{bastard, ferry,nanobook}. This has led to the possibility of building nanostructures like single electron transistors or single electron memories, which involve the transport of only a few electrons. In general, such devices consist in an active region (called the channel or the island) connecting two electrodes, known as the source and the drain, while the electrical potential in this active region can be tuned by a third electrode, the gate. In many applications, the performance of the device will depend on the possibility of controlling the electrons by acting on the gate voltage. 

At the nanometer scale, quantum effects such as interferences or tunneling become important and a quantum transport model is necessary. In this paper, we analyze the controllability of a simplified mathematical model of the quantum transport of electrons trapped in a two-dimensional device, typically a MOS field-effect transistor. The problem is modeled by a single Schr\"odinger equation, in a bounded domain $\Omega$ with homogeneous Dirichlet boundary conditions, coupled to the Poisson equation for the electrical potential. This work is a first step towards more realistic models. 
For instance, throughout the paper, the self-consistent  potential modeling interactions between electrons is either neglected, or (in the last section of the paper) considered as a small perturbation of the applied potential.

The control on this system is done through the boundary condition on the potential, on a part of the boundary modeling the gate. Degrees of freedom of the problem are the shape of the nanometric device and the position of the gate and its associated Dirichlet boundary conditions, modeling possible inhomogeneities: we prove that, generically  with respect to these degrees of freedom, the device is controllable.
We recall that Genericity is a measure of how frequently and robustly a property holds  with respect to some parameters. 

Controllability of general control-affine systems driven by the Schr\"odinger equation has been widely studied in the recent years.
The first positive controllability results for infinite-dimensional quantum systems have been established by local inversion theorems and the so-called \emph{return method} \cite{beauchard,beauchard-coron} (see also \cite{beauchard-laurent} for more recent results in this direction). 
Other results have been obtained by Lyapunov-function techniques and combinations with local inversion results \cite{beauchard-nersesyan,ito-k,Mirrahimi,nersesyan-morgan,Nersesyan,freresN} and by geometric control methods, using Galerkin or adiabatic  approximations \cite{Relax,metalemma,adiabatic-TAC,weakly-coupled,chambrion,CMSB}. 
Finally, let us conclude this necessarily incomplete list by mentioning that specific arguments have been developed to tackle physically relevant particular cases \cite{beauchard-bloch,ervedoza-puel,li-k}. 
Let us also recall that genericity of sufficient conditions for the controllability of the Schr\"odinger equation has been studied in \cite{mason-sigalotti,nersesyan-morgan,nersesyan-genericity,privat-sigalotti}.

Our analysis is based on the sufficient condition for approximate controllability obtained in \cite{Relax}, which requires a non-resonance condition on the spectrum of the internal Hamiltonian and a coupling property (the \textit{connectedness chain} property) on the external control field. Genericity is proved by global perturbations, exploiting the analytic dependence of the eigenpairs of the Schr\"odinger operator. 

\subsection{The quantum transport model}\label{sec:quantumTransport}

\subsubsection*{The unperturbed device}

Let us write a first model. In the following, $\Omega$ denotes a rectangle in the plane.

\vspace{3mm}

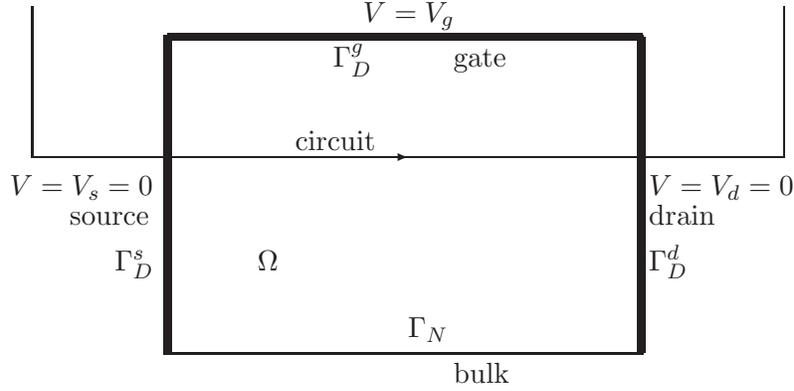
\begin{figure}[!ht]
 \begin{center}
\fbox{
\begin{picture}(6,4)
\put(1,1){$\Omega$}
\put(2.4,4.3){$V=V_g$}
\put(3.6,3.7){gate}
\put(-2.3,2){$V=V_s=0$}
\put(-1.5,1.6){source}
\put(6.2,2){$V=V_d=0$}
\put(6.2,1.6){drain}
\put(1.5,2.6){circuit}
\put(3.6,-0.5){bulk}
\put(2.0,3.7){$\Gamma_D^g$}
\put(-0.9,1){$\Gamma_D^s$}
\put(6.2,1){$\Gamma_D^d$}
\put(3,0.1){$\Gamma_N$}
\put(-2,2.5){\line(0,1){2}}
\put(-2,2.5){\vector(1,0){5}}
\put(3,2.5){\line(1,0){5}}
\put(8,2.5){\line(0,1){2}}
\linethickness{1mm}
\put(-0.2,-0.1){\line(0,1){4.2}}
\put(6.1,-0.1){\line(0,1){4.2}}
\put(-0.2,4.1){\line(1,0){6.3}}
\end{picture}}
\end{center}
\caption{Representation of the transistor}\label{figTransistor}
\end{figure}

 We assume without loss of generality that $\Omega=(0,\pi)\times (0,L)$ for some $L>0$, so that, with the notations of Figure \ref{figTransistor}, one has $\Gamma_D^s=\{0\}\times [0,L]$, $\Gamma_D^d=\{\pi\}\times [0,L]$, $\Gamma_N=[0,\pi]\times \{0\}$ and $ \Gamma_D^g=[0,\pi]\times \{L\}$. 
We set $\Gamma_D=\Gamma^s_D\cup \Gamma^d_D\cup \Gamma^g_D$.
In the whole paper, the notation $\frac{\partial}{\partial \nu}$ denotes the outward normal derivative.

We focus on the control problem
\begin{equation}\label{eq-schro1}
\left\{\begin{aligned}
i\partial_t\psi(t,x) &= -\Delta \psi(t,x) +V(t,x)\psi(t,x),\quad  && (t,x)\in  \mathbb{R}_+\times \Omega \\
-\Delta V(t,x) &= 0, && (t,x)\in  \mathbb{R}_+\times \Omega, \\
\psi (t,x)&=0, && (t,x)\in \mathbb{R}_+\times \partial \Omega,\\
V (t,x)&=\chi(x)V_g(t), && (t,x)\in \mathbb{R}_+\times \Gamma_D^g,\\
V (t,x)&=V_s=0, && (t,x)\in \mathbb{R}_+\times \Gamma_D^s, \\
V (t,x)&=V_d=0, && (t,x)\in \mathbb{R}_+\times   \Gamma_D^d,\\
\frac{\partial V}{\partial \nu }(t,x)&=0, && (t,x)\in \mathbb{R}_+\times \Gamma_N. 
\end{aligned}\right.
\end{equation}
The factor $\chi$ is an approximation of the constant function $\mathbb{1}_{Q(\Gamma^g_D)}$ that models spatial inhomogeneities, and is assumed to belong to
$$
\mathcal{C}_{0}^1(\Gamma_{D}^g)=\{\chi\in \mathcal{C}^1(\Gamma_{D}^g) \, : \, \chi(0,L)=\chi(\pi,L)=0\}.
$$
The vanishing condition on the boundary of the gate guarantees the continuity of the Dirichlet condition in the equation for $V$.  

Here, $\psi$ is the wave function of the electrons, satisfying the Schr\"odinger equation with the potential $V$. This potential solves the Poisson equation with a vanishing right-hand side, which means that we neglect the self-consistent electrostatic effects. In Section \ref{s:nonlinear}, as a generalization, we incorporate the self-consistent potential in the model as a perturbation of the applied potential $V$.

Let us comment on the boundary conditions. The wavefunction $\psi$ is subject to homogeneous Dirichlet boundary conditions, modelling the fact that the electrons are trapped in the device. For the potential $V$, the only nontrivial boundary condition is taken at the upper side of the rectangle $\Gamma_D^g$ where the gate is located. The applied grid voltage $t\mapsto V_g(t)$, with values in $[0,\delta]$ for some $\delta>0$ fixed throughout the paper, is seen as a \emph{control}, in the sense that the evolution of the system can be driven by its choice. At the source  and drain contacts $\Gamma_D^s$ and $\Gamma_D^d$, we impose homogeneous Dirichlet boundary conditions: we assume indeed for simplicity that $V_s=V_d=0$, the goal of this paper being to study the possibility of controlling by the gate. Finally, a Neumann boundary condition is imposed at the lower side $\Gamma_N$ of the rectangle, assumed in contact with the bulk where electrical neutrality holds.

The existence and uniqueness of mild solutions of \eqref{eq-schro1} in $\mathcal{C}^0(\R,L^2(\Omega,\C))$ for 
$V_g$ in $L^\infty(\R,[0,\delta])$ is then a consequence of general results for semilinear equations (see, for instance \cite{BMS} or \cite{pazy-book}).

\subsubsection*{Problem with shape inhomogeneities}

The problem above can be seen as an idealization, in the sense that the shape of the device is assumed to be perfectly rectangular.

Irregularities and inhomogeneities can be introduced in the model as follows. Let $Q$ be in
$$\mathrm{Diff}^1_0=\{Q:\R^2\to \R^2\mid Q \mbox{  orientation-preserving $\mathcal{C}^1$-diffeomorphism}\}$$
and $\chi$ be a function in 
\begin{equation}
\label{C20}
\mathcal{C}_0^1(Q(\Gamma^g_D))=\{\chi\in \mathcal{C}^1(Q(\Gamma^g_D))\,:\,\chi(Q(0,L))=\chi(Q(\pi,L))=0\}.
\end{equation}
Replacing $\Omega$ by $Q(\Omega)$, the resulting system writes
\begin{equation}\label{Qschro1}
\left\{\begin{aligned}
i\partial_t\psi(t,x) &= -\Delta \psi(t,x) +V(t,x)\psi(t,x),\quad  && (t,x)\in  \mathbb{R}_+\times Q(\Omega), \\
-\Delta V(t,x) &= 0, && (t,x)\in  \mathbb{R}_+\times Q(\Omega), \\
\psi (t,x)&=0, && (t,x)\in \mathbb{R}_+\times Q(\partial \Omega),\\
V (t,x)&=V_g(t) \chi(x), && (t,x)\in \mathbb{R}_+\times Q(\Gamma_D^g),\\
V (t,x)&=0, && (t,x)\in \mathbb{R}_+\times Q(\Gamma_D^s\cup \Gamma_D^d),\\
\frac{\partial V}{\partial \nu }(t,x)&=0, && (t,x)\in \mathbb{R}_+\times Q(\Gamma_N). 
\end{aligned}\right.
\end{equation}
We clearly have  $V(t,x)=V_g(t) V^{Q,\chi}_0(x)$ where $V_0^{Q,\chi}$ solves
\begin{equation}\label{V0t}
\left\{\begin{aligned}
-\Delta V_{0}^{Q,\chi}(x) &=0,\qquad && x\in Q(\Omega)\\
V_{0}^{Q,\chi}(x)&=\chi, && x\in Q(\Gamma_D^g)\\
V_{0}^{Q,\chi} (x)&=0, && x\in Q(\Gamma_D^s\cup \Gamma_D^d)\\
\frac{\partial V_{0}^{Q,\chi}}{\partial \nu }(x)&=0, && x\in Q(\Gamma_N).
\end{aligned}
\right.
\end{equation}
As for the unperturbed system, mild solutions of \eqref{Qschro1} in $\mathcal{C}^0(\R,L^2(Q(\Omega,\C))$ exist and are unique  for 
$V_g$ in $L^\infty(\R,[0,\delta])$.

\subsection{Action of the grid voltage on the system}

\subsubsection*{A control approach}

Our aim is to understand to what extent the system can be manipulated through the grid voltage. In this perspective, the time-varying parameter $V_g(\cdot)$ is seen as a control law and the objective is to characterize the controllability properties of the resulting system. 

\begin{definition}
We say that the control system  \eqref{Qschro1} is \emph{approximately controllable} if, 
for every $\psi_0,\psi_1\in L^2(Q(\Omega),\C)$ with unit norm and every $\varepsilon>0$, there exist a positive time $T$ and a control $V_g\in L^\infty([0,T],[0,\delta])$  such that 
the solution $\psi$ of  \eqref{Qschro1} with initial condition  $\psi(0)=\psi_0$ satisfies $\|\psi(T)-\psi_1\|_{L^2(Q(\Omega))}<\varepsilon$. 
\end{definition}

Notice that, for  quantum control systems with bounded control operators, exact controllability\footnote{System  \eqref{Qschro1}  would be \emph{exactly controllable} if
for every $\psi_0,\psi_1\in L^2(Q(\Omega),\C)$ with unit norm, there existed a positive time $T$ and a control $V_g\in L^\infty([0,T],[0,\delta])$  such that $\psi(T)=\psi_1$, where 
 $\psi$ denotes the solution  of  \eqref{Qschro1} corresponding to $V_g$ with initial condition $\psi(0)=\psi_0$. 
}  cannot be expected (see \cite{BMS,turinici}). This justifies our 
choice of approximate controllability as a notion of arbitrary maneuverability of the system.  
Other possible notions of controllability considered in the literature are exact controllability between smooth enough wavefunctions (see \cite{beauchard-coron,beauchard-laurent}) or exact controllability in infinite time (see \cite{freresN}).

The issue of determining whether \eqref{Qschro1} is approximately controllable for a given pair $(Q,\chi)$ seems a 
difficult task in general, since the known sufficient criteria for approximate controllability require a fine knowledge of the spectral properties of the operators involved (see Section \ref{s-general}).
Instead, our main goal is to 
study the controllability properties of the model which hold true \emph{generically} with respect to the diffeomorphism $Q$ and the boundary condition $\chi$. 
Genericity is a measure of how often and with which degree of robustness a property holds. 
More precisely, a property described by a boolean function $P:X\to \{0,1\}$ is said to be \textit{generic} in a Baire space $X$ if there exists a residual set\footnote{i.e. the intersection of countably many open and dense subsets.} $Y\subset X$ such that every $x$ in $Y$ satisfies the property $P$, that is, $P(x)=1$. Recall that a residual set is in particular dense in $X$.

\subsubsection*{Genericity results with respect to $\chi$ and $(Q,\chi)$}\label{enonces}

We are now ready to state our two main results. First consider the problems of the form \eqref{Qschro1} where $Q=\Id$, for which the genericity of the controllability is considered only with respect to variations of the boundary condition $\chi$ on the grid $\Gamma_D^g$. We allow $\chi$ to vary within the class  
$\mathcal{C}_0^1(\Gamma^g_D)$ defined in \eqref{C20}, whose metric is complete, making it a Baire space.

We have the following genericity result. 
\begin{theorem}\label{grille-tangent}
Let $L^2\not\in \pi^2\mathbb{Q}$.
For $Q=\Id$ and a generic $\chi$ in  $\mathcal{C}_0^1(\Gamma^g_D)$, the control problem \eqref{Qschro1} is approximately controllable.
\end{theorem}

Consider now the entire class of problems of the form \eqref{Qschro1}. In order to endow it with  a topological structure, 
we identify  \eqref{Qschro1} with the triple  $(Q(\Omega),Q(\Gamma_D^g),\chi)$. 
The family of problems is then given by  
\begin{equation}\label{def-P}
\mathcal{P}=\{(Q(\Omega),Q(\Gamma_D^g),\chi)\mid 
(Q(\Omega),Q(\Gamma_D^g))\in \Sigma\quad \mbox{and}\quad 
\chi\in \mathcal{C}_0^1(Q(\Gamma^g_D))\},
\end{equation}
where
$$\Sigma=\{{
(Q(\Omega),Q(\Gamma_D^g))}\mid Q\in \mathrm{Diff}^1_0\}.$$ 
The metric induced by that of $\mathcal{C}^1$-diffeomorphisms 
and by the $\mathcal C^1$ topology on $\mathcal{C}^1_0(Q(\Gamma^g_D))$ 
makes $\mathcal{P}$ complete
 (\cite{micheletti}). 
In particular, $\mathcal{P}$ is a Baire space.

\begin{theorem}\label{grille-deformation}
For a generic element of $\mathcal{P}$, the control problem \eqref{Qschro1} is approximately controllable.
\end{theorem}

The proofs of Theorems \ref{grille-tangent} and \ref{grille-deformation} can be found in Sections \ref{sec:grille-tangent} and \ref{sec:grille-deformation}, respectively. 
They are based on a general sufficient condition for controllability proved in \cite{Relax} and recalled in Section~\ref{s-general} below. In a nutshell, such a condition is based, on the one hand, on a nonresonance property of the spectrum of the Schr\"odinger operator and, on the other hand, on a coupling property for the interaction term (see the notion of connectedness chain introduced in Definition~\ref{d-chain}). These properties are expressed as a countable number of open conditions. Their density is proved through a global analytic propagation argument. 

In Section \ref{s-generalizations}, we present two generalizations of these results, motivated by the applications. First, in Subsection \ref{s:partial}, we consider a situation where the gate only partially covers the upper side of the rectangle domain. Then, in Subsection \ref{s:nonlinear}, we take into account in our model the self-consistent electrostatic Poisson potential, as a perturbation of the applied potential $V$.

\section{Proof of the genericity results}

\subsection{General controllability conditions for bilinear quantum systems}\label{s-general}

We recall in this section a general approximate controllability result for bilinear quantum systems obtained in \cite{Relax}.

Let $\cal H$ be a complex Hilbert space with scalar product $\langle \cdot ,\cdot \rangle$ and  $A,B$ be two linear skew-adjoint operators on $\cal H$. Let $B$ be bounded and denote by $D(A)$ the domain of $A$. 
Consider the 
controlled equation
\begin{equation} \label{eq:main}
\frac{d\psi}{dt}(t)=(A+u(t)B) \psi(t), \quad u(t) \in [0,\delta],
\end{equation}
with $\delta>0$. We say that $A$ 
 satisfies assumption $(\mathfrak{A})$ 
if there exists an orthonormal basis 
$(\phi_k)_{k \in \mathbb{N}}$ of $\cal H$
made of eigenvectors of $A$
whose associated 
eigenvalues $(i \lb_{k})_{k \in \mathbb{N}}$ are all simple.

\begin{definition}\label{d-chain}
A subset $S$ of
$\mathbb{N}^2$ \emph{couples} two levels $j,k$ in $\mathbb{N}$
if
there exists a finite sequence $\big ((s^{1}_{1},s^{1}_{2}),\ldots,(s^{p}_{1},s^{p}_{2}) \big )$
in $S$ such that
\begin{description}
\item[$(i)$] $s^{1}_{1}=j$ and $s^{p}_{2}=k$;
\item[$(ii)$] $s^{j}_{2}=s^{j+1}_{1}$ for every $1 \leq j \leq p-1$.
\end{description}
$S$ is called a \emph{connectedness chain} if $S$  couples every pair of levels in $\mathbb{N}$.

$S$ is a \emph{non-resonant connectedness chain for $(A,B,\Phi)$} if it is a connectedness chain,  
$\langle  \phi_{j}, B \phi_{k}\rangle \neq 0$ for every $ (j,k)\in S$, and  
$\lb_{s_1}-\lb_{s_2}\neq \lb_{t_1}-\lb_{t_2}$ for every
 $(s_1,s_2)\in S$ with $s_1\ne s_2$ and every $(t_{1},t_{2})$ in
$\mathbb{N}^2\setminus\{(s_1,s_2)\}$ such that $\langle \phi_{t_{1}}, B \phi_{t_{2}}\rangle  \neq 0$.
\end{definition}

\begin{theorem}[\cite{Relax}]\label{THE_Control_collectively}\label{t-relax}
 Let
  $A$ satisfy $(\mathfrak{A})$ and let $\Phi=(\phi_k)_{k \in \mathbb{N}}$ be an orthonormal basis of eigenvectors of $A$. 
 If there exists a non-resonant connectedness chain for $(A,B,\Phi)$
then
\eqref{eq:main} is approximately controllable.
\end{theorem}

\begin{remark}\label{strongly-non-resonant}
The simplicity of the spectrum 
required in Definition~\ref{d-chain} is not necessary. The construction in \cite{Relax}  is indeed  
slightly more general and we refer to that paper and \cite{metalemma} for further details. 

We also recall that a similar result based on a
 stronger requirement has been proposed in \cite{CMSB}. In that paper,  the spectrum of the operator $A$ 
was asked  to be \emph{non-resonant}, in the sense that  every nontrivial finite linear combination with rational coefficients of its eigenvalues was asked to be nonzero.
\end{remark}

\begin{remark}
The statement of Theorem~\ref{t-relax} could be strengthened, according to the results in \cite{Relax}, in two other directions:  first, the controllability could be extended beyond single wavefunctions, towards ensembles (controllability in the sense of density matrices and simultaneous controllability); second,  
unfeasible trajectories in the unit  sphere of $\cal H$ turn out to be trackable (i.e., they can be followed approximately with arbitrarily precision by admissible ones) at least when the modulus (but not the phase) of the components of the wavefunction are considered. 
Moreover, the proof of Theorem~\ref{t-relax} given in \cite{Relax} is constructive, leading to a control design algorithm based on the knowledge of the spectrum of the operator $A$ (see also \cite{chambrion} for an alternative construction).
\end{remark}

\begin{remark}\label{controllabilityH1}
Another consequence of the Lie--Galerkin approach behind Theorem~\ref{t-relax} is that the 
conclusions  of Theorems \ref{grille-tangent} and \ref{grille-deformation}   could be strengthened by stating approximate
controllability in stronger topologies.
The key point is that approximate controllability can be obtained by requiring, in addition, that the total variation and the $L^1$ norm of the control law are bounded uniformly with respect to the tolerance (see \cite{Relax,chambrion}). Proposition 3 in \cite{thomas-CPDE} then implies that, for an initial and final conditions $\psi_0,\psi_1\in H^2(\Omega)$, for every tolerance $\varepsilon>0$, there exists a control steering $\psi_0$ to an $\varepsilon$-neighbourhood of $\psi_1$ for the  $L^2$-norm, while satisfying a uniform bound (independent of $\varepsilon$) for the $H^2$-norm. An interpolation argument allows to 
conclude that, for $\xi\in (0,2)$,  $\psi_0$ can be steered $\varepsilon$-close to $\psi_1$ in the $H^{\xi}$-norm.
 \end{remark}

\subsection{Preliminary steps of the proofs}\label{s:prelim}

The proofs of Theorems~\ref{grille-tangent} and \ref{grille-deformation} are based on the idea of \emph{propagating} sufficient controllability conditions using analytic perturbations (\cite{hillairet-judge,mason-sigalotti,privat-sigalotti}). This is possible since the general controllability criterion for quantum systems seen in the previous section can be seen as a countable set of nonvanishing scalar conditions.

More precisely, let us denote by
$\Lambda(Q(\Omega))$ the spectrum of the Laplace--Dirichlet operator on $Q(\Omega)$ and, for every
$(Q(\Omega),Q(\Gamma_D^g),\chi)\in \mathcal{P}$ such that $\Lambda(Q(\Omega))$ is simple (i.e., each eigenvalue is simple), 
define
\begin{equation}\label{eq:S}
S(Q(\Omega),Q(\Gamma_D^g),\chi)=\left\{(k,j)\in \mathbb{N}^2\mid \int_{Q(\Omega)} V^{Q,\chi}_0(x)\phi_k(x)\phi_{j}(x)dx \neq 0\right\},
\end{equation}
where $\{\phi_j\}_{j\in\mathbb{N}}$ is a Hilbert basis of eigenfunctions of the Laplace--Dirichlet operator on $Q(\Omega)$, ordered following the 
growth of the corresponding eigenvalues. 

Theorem~\ref{grille-deformation} is proved by 
applying Theorem~\ref{t-relax} with $A$ the Laplace--Dirichlet operator on $Q(\Omega)$ multiplied by $i$ and $B$ the multiplicative operator defined by $B\psi=-i V^{Q,\chi}_0\psi$.  
We then show that both sets
$$
\mathcal{P}_1=\{(Q(\Omega),Q(\Gamma_D^g),\chi)\in \mathcal{P}\mid \Lambda(Q(\Omega)) \mbox{ non-resonant}\},$$
where 
the notion of non-resonant spectrum is the one introduced in Remark~\ref{strongly-non-resonant}, and 
$$\mathcal{P}_2=\{(Q(\Omega),Q(\Gamma_D^g),\chi)\in \mathcal{P}\mid \Lambda(Q(\Omega)) \mbox{ simple,  $S(Q(\Omega),Q(\Gamma_D^g),\chi)$ connectedness chain}\}
$$
are residual in $\mathcal{P}$. Their intersection is therefore residual as well (it is itself the intersection of countably many open dense sets). The following result resumes these considerations. 

\begin{proposition}\label{p:P1P2}
If $\mathcal{P}_1$ and $\mathcal{P}_2$ are residual then
 the control problem \eqref{Qschro1} is approximately controllable for a generic element of $\mathcal{P}$.
\end{proposition}

The situation is slightly different for the proof of Theorem~\ref{grille-tangent}, since the fact that $Q=\Id$ prevents 
$ \Lambda(Q(\Omega))=\Lambda(\Omega)$ from being non-resonant. Recall that $(0,\delta)$ is the interval of admissible control values (see Section \ref{sec:quantumTransport}).
We are then led to rewrite, for every $\rho\in[0,\delta)$, equation \eqref{Qschro1}  
in the case $Q=\Id$ as 
\begin{equation}\label{Qschro1-eps}
\left\{\begin{aligned}
i\partial_t\psi(t,x) &= (-\Delta+\rho V_0^{\Id,\chi}(x)) \psi(t,x) +(V_g(t)-\rho) V_0^{\Id,\chi}(x)\psi(t,x), && (t,x)\in  \mathbb{R}_+\times \Omega, \\
\psi (t,x)&=0, && (t,x)\in \mathbb{R}_+\times \partial \Omega.
\end{aligned}\right.
\end{equation}

We apply  Theorem~\ref{t-relax} to \eqref{Qschro1-eps} with $A= -i(-\Delta+\rho V_0^{\Id,\chi}\Id)$ on $\Omega$ (with Dirichlet boundary conditions) and $B=-iV^{\Id,\chi}_0\Id$.
In analogy to the notation  introduced above, let
$$\mathcal{P}_{1,\mathrm{BC}}^\rho =\{\chi\in \mathcal{C}^1_0(\Gamma_D^g)\mid 
\mbox{the spectrum of $-\Delta+\rho V_0^{\Id,\chi}\Id$ is weakly non-resonant}\},$$
where a sequence $(\lambda_n)_{n\in \mathbb{N}}$ is said to be \emph{weakly non-resonant} if
$\lambda_{s_1}-\lambda_{s_2}\ne \lambda_{t_1}-\lambda_{t_2}$ for every $(s_1,s_2),(t_1,t_2)\in \mathbb{N}^2$ with 
$s_1\ne s_2$ and $(s_1,s_2)\ne (t_1,t_2)$.

Moreover, let
\begin{align*}
\mathcal{P}_{2,\mathrm{BC}}^\rho =&\{\chi\in \mathcal{C}_0^1(\Gamma_D^g)\mid 
\mbox{the spectrum of $-\Delta+\rho V_0^{\Id,\chi}\Id $ is simple}\\
 &\mbox{ and $S_\rho(\chi)$ is a connectedness chain}\}
 \end{align*}
where
\begin{equation}\label{eq:Srho}
S_\rho(\chi)=\left\{(k,j)\in \mathbb{N}^2\mid \int_{\Omega} V^{\Id,\chi}_0(x)\phi_{k,\rho}(x)\phi_{j,\rho}(x)dx \neq 0\right\}
\end{equation}
and $\{\phi_{j,\rho}\}_{j\in\mathbb{N}}$ is a Hilbert basis of eigenfunctions of $-\Delta+\rho V_0^{\Id,\chi}\Id$, ordered following the 
growth of the 
corresponding eigenvalues.

System~\eqref{Qschro1-eps} is approximately controllable if $\chi\in \mathcal{P}_{1,\mathrm{BC}}^\rho\cap \mathcal{P}_{2,\mathrm{BC}}^\rho$ for some $\rho\in (0,\delta)$. 

Theorem~\ref{grille-tangent} is then proved 
through the following proposition, playing the role of Proposition~\ref{p:P1P2} in the case $Q=\Id$. 

\begin{proposition}\label{p:P1BCP2BC}
Let $L^2\not\in \pi^2\mathbb{Q}$ and $Q=\Id$.
If there exists $\rho\in (0,\delta)$ such that $\mathcal{P}_{1,\mathrm{BC}}^\rho$ and $\mathcal{P}_{2,\mathrm{BC}}^\rho$ are residual then  the control problem \eqref{Qschro1} is approximately controllable for a generic $\chi$ in  $\mathcal{C}_0^1(\Gamma^g_D)$.
\end{proposition}

A crucial tool for proving that the sets introduced above are residual is the following proposition, stating that  $V_0^{Q,\chi}$ is  analytic with respect to $Q$ and $\chi$.

\begin{proposition}\label{p:analytic}
Let $I$ be an open interval and $I\ni t\mapsto (Q_t,\varphi_t)$ be an analytic curve in the product of $\mathrm{Diff}^1_0$ with the space $\mathcal{C}_0^1(\Gamma_D^g)$ defined in \eqref{C20}.
Denote by 
$\chi_t$ the composition $\varphi_t\circ Q_t^{-1}$ and by 
$V_{0,t}$ the function $V_0^{Q_t,\chi_t}$ defined as in \eqref{V0t}.  
Then $t\mapsto V_{0,t}\circ Q_t$ is an analytic curve in $H^1(\Omega)$.
\end{proposition}

The proof of the proposition is given in next section. One important consequence for our argument is the following corollary. 

\begin{corollary}\label{propagation}
Let $\hat{\mathcal{P}}$ be one of the sets $ \mathcal{P}_1$, $ \mathcal{P}_2$,
$ \mathcal{P}^\rho_{1,\mathrm{BC}}$, $ \mathcal{P}^\rho_{2,\mathrm{BC}}$.
If $\hat{\mathcal{P}}$ is nonempty, then 
$\hat{\mathcal{P}}$ is residual.
Moreover, if $ \mathcal{P}^\rho_{j,\mathrm{BC}}$ is nonempty for $j\in\{1,2\}$ and $\rho\in[0,\delta)$, then 
$ \mathcal{P}^{\rho'}_{j,\mathrm{BC}}$ is nonempty (and hence dense) for almost all $\rho'\in [0,\delta)$. 
\end{corollary}
\begin{proof}
In order to avoid redundancies, we prove the corollary only in the case $\hat{\mathcal{P}}=\mathcal{P}_2$. The proof can be easily adapted to the other cases. 

Let us first prove that $ \mathcal{P}_2$ is the intersection of countably many open sets. 
We claim that $ \mathcal{P}_2=\cap_{n\in \mathbb{N}}\mathcal{A}_n$, where
$\mathcal{A}_n$ is the set of triples $(Q(\Omega),Q(\Gamma_D^g),\chi)\in\mathcal{P}$ such that 
 the first $n$ eigenvalues of the Laplace--Dirichlet operator on $Q(\Om)$ are simple 
  and there exist 
 $r\in \mathbb{N}$ and $r$ other simple eigenvalues of $\lambda_{k_1}, \dots,\lambda_{k_r}$ such that the matrix
\begin{equation}\label{submat}
\left(\int_{Q(\Om)}\phi_j\phi_l V_0^{Q,\chi}\right)_{j,l\in\{1,\dots,n\}\cup\{k_1,\dots,k_r\} }
\end{equation}
 is connected\footnote{We recall 
 that a $m\times m$ matrix $C =(c_{jl})_{j,l=1}^m$  is said to be \emph{connected} if for every pair of indices 
 $j,l=1,\dots,m$ there 
there exists a finite sequence 
$j_1,\dots,j_w\in\{1,\dots,m\}$ such that $c_{jj_1}c_{j_1j_2}\cdots c_{j_{w-1}j_{w}}c_{j_w l}\ne 0$.
The set $\{(j,l)\mid c_{jl}\ne 0\}$ is said to be a \emph{connectedness chain} for $C$.}, where each $\phi_j$ is an eigenfunction corresponding to $\lambda_j$.
 It is clear that an element of $\cap_{n\in \mathbb{N}}\mathcal{A}_n$ is in $\mathcal{P}_2$, since its corresponding spectrum is simple and a connectedness chain is given by the union of all the connectedness chains for the matrices of the type \eqref{submat}.
 Conversely, if $(Q(\Omega),Q(\Gamma_D^g),\chi)\in\mathcal{P}_2$, then there exists a bijection 
 $\xi:\mathbb{N}\to \mathbb{N}$ such that each matrix
$$ \left(\int_{Q(\Om)}\phi_{\xi(j)}\phi_{\xi(l)} V_0^{Q,\chi}\right)_{j,l\in\{1,\dots,n\} }$$
is connected (see \cite[Remark 4.2]{mason-sigalotti}). 
Given $n\in\mathbb{N}$, let $N$ be such that $\xi(\{1,\dots,N\})\supset \{1,\dots,n\}$. Then, taking $r=N-n$ and $\{\lambda_{k_1},\dots,\lambda_{k_r}\}=\{\lambda_{\xi(1)},\dots,\lambda_{\xi(N)}\}\setminus \{\lambda_1,\dots,\lambda_n\} $,
we have that $(Q(\Omega),Q(\Gamma_D^g),\chi)\in\mathcal{A}_n$. 

Since each $\mathcal{A}_n$ is open (by continuity of the eigenpairs corresponding to simple eigenvalues),  we have proved that $ \mathcal{P}_2$ is the intersection of countably many open sets.

Let us now show that $\mathcal{P}_2$ is dense if it is nonempty.  
Fix $(Q(\Om),Q(\Gamma^g_D),\chi)\in \mathcal{P}_2$ and let
$$\bar S=S(Q(\Om),Q(\Gamma^g_D),\chi).$$
Let 
$I\ni t\mapsto (Q_t,\varphi_t)$ be an analytic curve in the product $\mathrm{Diff}^1_0\times \mathcal{C}_0^1(\Gamma_D^g)$ and assume that 
there exists $t_0\in I$ such that $Q_{t_0}=Q$ and $\varphi_{t_0}\circ Q=\chi$.
According to Rellich's theorem (see \cite{kato,rellich}), there exists $
I\ni t\mapsto (\lambda_j(t),\phi_{j}(t))_{j\in \mathbb{N}}$ such that 
$(\lambda_j(t),\phi_{j}(t))_{j\in \mathbb{N}}$ is a complete family of eigenpairs of the Laplace--Dirichlet operator on $Q_t(\Om)$ for every $t\in I$, with 
$I\ni t\mapsto \lambda_j(t)$ and $I\ni t\mapsto \phi_j(t)\circ Q_t$ analytic  in $\R$ and 
in $L^2(\Om,\R)$, respectively, for every $j\in\mathbb{N}$. 

Proposition~\ref{p:analytic} implies that for every $j,k\in\mathbb{N}$, the function 
$$t\mapsto \int_{Q_t(\Om)}\phi_j(t)\phi_k(t) V_0^{Q_t,\varphi_t\circ Q_t^{-1}}$$
is analytic on $I$. 
Moreover, the spectrum $\Lambda(Q_t(\Om))$ is simple for almost every $t\in I$. 

We can assume that  the sequence $(\lb_j(t_0))_{j\in\mathbb{N}}$ is (strictly) increasing.  
For every $t\in I$ such that  $\Lambda(Q_t(\Om))$ is simple, there exists $\xi_t:\mathbb{N}\to\mathbb{N}$ bijective such that $(\lb_{\xi_t(j)}(t))_{j\in\mathbb{N}}$ is  increasing. By analyticity of 
$t\mapsto \int_{Q_t(\Om)}\phi_j(t)\phi_k(t) V_0^{Q_t,\varphi_t\circ Q_t^{-1}}$ for each $(j,k)\in \bar  S$, we have that
$$\{(\xi^{-1}_t(j),\xi^{-1}_t(k))\mid (j,k)\in \bar S\}\subset S(Q_t(\Om),Q_t(\Gamma_D^g),\varphi_t\circ Q_t^{-1})$$
 for almost every $t\in I$. Since for every bijection $\hat \xi:\mathbb{N}\to\mathbb{N}$ the set 
$\{(\hat\xi(j),\hat \xi(k))\mid (j,k)\in \bar S\}$ is a connectedness chain, we conclude that for almost every $t\in I$, 
$S(Q_t(\Om),Q_t(\Gamma_D^g),\varphi_t\circ Q_t^{-1})$ is a connectedness chain. 
Hence, for almost every $t\in I$, $(Q_t(\Om),Q_t(\Gamma_D^g),\varphi_t\circ Q_t^{-1})\in \mathcal{P}_2$.

We conclude on the density of $\mathcal{P}_2$ by considering all analytic curves $t\mapsto (Q_t,\varphi_t)$ passing through $(Q,\chi)$, since each pair of elements of $\mathcal{P}_2$ can be connected by an analytic path. The second part of the statement is proved by analogous analyticity considerations with respect to the parameter $\rho$.   
\end{proof}

\subsection{Proof of Proposition \ref{p:analytic}}

Denote by $\hat \varphi_t$ the extension of $\varphi_t$ on $\overline \Omega$ which is constant on every vertical segment.
Then  $t\mapsto \hat \varphi_t$ is an analytic curve in $\mathcal{C}^1(\overline \Omega)$ with
$\hat \varphi_t\equiv 0$ on $\Gamma_D^s\cup \Gamma_D^d$.

Define $\hat \chi_t=\hat \varphi_t \circ Q_t^{-1}$ and let $\hat V_{0,t}=V_{0,t}-\hat \chi_t$. 
Notice that $\hat V_{0,t}$ is a solution to the problem 
\begin{equation}\label{hatV0t}
\left\{\begin{aligned}
-\Delta \hat V_{0,t}(x) &=\Delta \hat \chi_t\,,\quad  && x\in Q_t(\Omega),\\
\hat V_{0,t}(x)&=0, && x\in Q_t(\Gamma_D),\\
\frac{\partial \hat V_{0,t}}{\partial \nu }(x)&=0, && x\in Q_t(\Gamma_N).
\end{aligned}
\right.
\end{equation}
Equivalently,
$$\int_{Q_t(\Omega)}\nabla \hat V_{0,t}(x)\cdot \nabla \phi(x)\,dx=\int _{Q_t(\Omega)}\Delta \hat \chi_{t}(x) \phi(x)\,dx,$$ 
for every $\phi\in H^1_{0,Q_t(\Gamma_D)}(Q_t(\Omega))$, where 
$$H^1_{0,Q_t(\Gamma_D)}(Q_t(\Omega))=\{\phi\in H^1(Q_t(\Omega))\mid \phi=0\mbox{ on }Q_t(\Gamma_D)\}.$$

Fix $t_0\in I$ and notice that, for every $t\in I$, 
$$H^1_{0,Q_t(\Gamma_D)}(Q_t(\Omega))=\{\phi\circ Q_{t_0}\circ Q_t^{-1}\mid \phi\in H^1_{0,Q_{t_0}(\Gamma_D)}(Q_{t_0}(\Omega))\}.$$ 
Set $R_t=Q_{t}\circ Q_{t_0}^{-1}$. By the standard change of coordinates 
formula, 
$$ \int_{Q_{t_0}(\Omega)} ((D R_t^T)^{-1}\nabla  W_{t})\cdot ((D R_t^T)^{-1} \nabla \phi)J_t=\int_{Q_{t_0}(\Omega)} (\Delta \hat \chi_t\circ R_t)\phi \,J_t
$$
for every $\phi\in H^1_{0,Q_{t_0}(\Gamma_D)}(Q_{t_0}(\Omega))$,
where 
$D R_t$ and $D R_t^T$ are,  respectively, the Jacobian matrix of $R_t$ and its transpose, while $W_{t}=\hat V_{0,t}\circ R_t$ and $J_t=\det (D R_t)$.

In other words, $(t,W_{t})$ is the solution of $F(t,W_{t})=0\in H^{-1}_{0,Q_{t_0}}(Q_{t_0}(\Omega))$, where $H^{-1}_{0,Q_{t_0}}(Q_{t_0}(\Omega))$ stands for the dual space of $H^{1}_{0,Q_{t_0}}(Q_{t_0}(\Omega))$ with respect to the pivot space $L^2(Q_{t_0}(\Omega))$, with 
\begin{align*}F(t,W)&=-\mathrm{div}(A_t \nabla W)-(\Delta\hat\chi_{t}\circ R_t)J_t,\\
A_t&=J_t (DR_t)^{-1}(DR_t^T)^{-1}.
\end{align*}
The analyticity of $W_{t}$ with respect to $t$ follows by the implicit function theorem, since $F$ is analytic 
from $I\times H^1_{0,Q_{t_0}(\Gamma_D)}(Q_{t_0}(\Omega))$ into $H^{-1}_{0,Q_{t_0}}(Q_{t_0}(\Omega))$ and the operator $D_W F(t_0,W_{t_0})$ is an isomorphism of $H^1_{0,Q_{t_0}(\Gamma_D)}(Q_{t_0}(\Omega))$ into $H^{-1}_{0,Q_{t_0}}(Q_{t_0}(\Omega))$. 
Indeed, by linearity of $F$ with respect to $W$ and because $R_{t_0}$ is the identity, $D_W F(t_0,W_{t_0})Z$ is nothing else that $-\Delta Z$, which is an isomorphism from $H^{1}_{0,Q_{t_0}}(Q_{t_0}(\Omega))$ to $H^{-1}_{0,Q_{t_0}}(Q_{t_0}(\Omega))$, by Lax-Milgram's lemma.

This concludes the proof of Proposition~\ref{p:analytic}.

\subsection{Proof of Theorem~\ref{grille-tangent}}\label{sec:grille-tangent}

Notice that the assumption $L^2\not\in \pi^2\mathbb{Q}$ guarantees that the spectrum of 
the Laplace--Dirichlet operator on $\Omega$ is simple.

According to Proposition~\ref{p:P1BCP2BC} and Corollary~\ref{propagation}, we are left to prove that there exist $\rho_1,\rho_2\in [0,\delta)$ such that
$\mathcal{P}^{\rho_1}_{1,\mathrm{BC}}$ and $\mathcal{P}^{\rho_2}_{2,\mathrm{BC}}$ are nonempty. 
The second part of the statement of Corollary~\ref{propagation}, indeed, implies then that there exists $\rho\in(0,\delta)$ such that 
$\mathcal{P}^{\rho}_{1,\mathrm{BC}}\cap \mathcal{P}^{\rho}_{2,\mathrm{BC}}$ is residual.

The proof that $\mathcal{P}^{\rho}_{1,\mathrm{BC}}$ is nonempty for some $\rho\in(0,\delta)$ is made in Section~\ref{P1rho}, while it is shown in Section~\ref{P2rho} that $\mathcal{P}^{0}_{2,\mathrm{BC}}$ is nonempty.

\subsubsection{There exists $\rho\in (0,\delta)$ such that $\mathcal{P}^{\rho}_{1,\mathrm{BC}}$ is nonempty}\label{P1rho}

Let
$n$ denote a positive integer and
 $\chi_n\in \mathcal{C}^\infty(\Gamma_D^g)$
be defined by 
\begin{equation}\label{chin}
\chi_n(x_1,L)= \cosh (n  L)\sin \left({n x_1}\right).
\end{equation}
Notice that in this case the solution $V_0^{\Id,\chi_n}$ of \eqref{V0t}
 is explicitly given by
$$
V_0^{\Id,\chi_n}(x)=\sin(nx_1)\cosh(nx_2),\qquad x=(x_1,x_2)\in \Omega.
$$

\begin{proposition}\label{noname}
Let $L^2\not\in \pi^2 \mathbb{Q}$. If $n$ is odd, then $\chi_n$ is in $\mathcal{P}_{1,\mathrm{BC}}^\rho$ 
 for almost every $\rho\in (0,\delta)$.
 \end{proposition}
 \begin{proof}
The eigenpairs of the Laplace--Dirichlet operator on $ \Omega$ are naturally parameterized over $\mathbb{N}^2$ as follows: for every ${\bf j}=(j_1,j_2)\in\mathbb{N}^2$, let
$$\lambda_{\bf j}={j_1^2}  + j_2^2\frac{\pi^2}{L^2},\qquad \phi_{\bf j}(x)=\frac{2}{\sqrt{\pi  L}}\sin(j_1 x_1)\sin\left(j_2\frac{\pi}{ L}\right).$$

For every $\rho\in [0,\delta)$, denote by $(\lambda_{\bf j}(\rho))_{{\bf j}\in\mathbb{N}^2}$ the spectrum 
of $-\Delta+\rho V^{\Id,\chi_n}_0\Id$. Each function $\lambda_{\bf j}(\cdot)$ can be chosen to be analytic on $[0,\delta)$, with $\lambda_{\bf j}(0)=\lambda_{\bf j}={j_1^2}  + j_2^2\frac{\pi^2}{L^2}$.

Let us evaluate the derivative of each $\lambda_{\bf j}(\rho)$ at  $\rho=0$. 
Denote $\alpha_{\bf j}=\left.\frac{d\lambda_{\bf j}}{d\rho}\right\vert_{\rho=0}$. 
Recall that the derivative of the eigenvalues can be computed according to the formula 
\begin{equation}\label{derV}
\alpha_{\bf j}=\int_{\Omega}V_{0}^{\Id,\chi_n}(x_{1},x_{2})\phi_{\bf j}(x_{1},x_{2})^2\, dx_{1}\, dx_{2}
\end{equation}
(see, for instance, \cite{Henrot}).

We assume that
\begin{align}
\lambda_{\bf j}(0)-\lambda_{\bf k}(0)&=\lambda_{\bf j'}(0)-\lambda_{\bf k'}(0),\label{lambda}\\
\alpha_{\bf j}-\alpha_{\bf k}&=\alpha_{\bf j'}-\alpha_{\bf k'}\label{alpha}
\end{align}
for some $\bf j,\bf k,\bf j',\bf k'\in \mathbb{N}^2$ with $(\bf k,\bf j)\ne (k',j')$ and we show that $(\bf k,k')=(j,j')$. 
By analyticity we then have that 
$\lambda_{\bf j}(\rho)-\lambda_{\bf k}(\rho)=\lambda_{\bf j'}(\rho)-\lambda_{\bf k'}(\rho)$ only for isolated values of $\rho\in(0,\delta)$ and the proposition follows by the countability of $\mathbb{N}^2\times \mathbb{N}^2$.

According to \eqref{lambda}, we have
$$
\frac{j_{1}^{2}}{\pi^2}+\frac{j_{2}^2}{L^2}-\frac{k_{1}^2}{\pi^2}-\frac{k_{2}^2}{L^2}=\frac{j_{1}'^{2}}{\pi^2}+\frac{j_{2}'^2}{L^2}-\frac{k_{1}'^2}{\pi^2}-\frac{k_{2}'^2}{L^2}.
$$
Since $L^2\notin\pi^2 \mathbb{Q}$, we get
\begin{eqnarray}
j_{1}^{2}-k_{1}^2& =& j_{1}'^{2}-k_{1}'^2,\label{ind1}\\
j_{2}^{2}-k_{2}^2& =& j_{2}'^{2}-k_{2}'^2.\label{ind2}
\end{eqnarray}
Computing \eqref{derV} 
using the expression 
$$
V_{0}^{\Id,\chi_n}(x_{1},x_{2})=\sin (nx_{1})\cosh (nx_{2}),
$$
we have
\begin{eqnarray*}
\alpha_{\bf j}  & = & \frac{4}{L\pi}\int_{0}^\pi \sin (nx_{1})\sin (j_{1}x_{1})^2\, dx_{1} \int_{0}^L \cosh (nx_{2})\sin \left(\frac{j_{2}\pi x_{2}}{L}\right)^2\, dx_{2} \\
 & = & -\frac{32L\pi\sinh(nL)}{n^2}\frac{j_{1}^2}{4j_{1}^2-n^2}\frac{j_{2}^2}{(2\pi)^2j_{2}^2+n^2}.
\end{eqnarray*}
Hence, we can rewrite  \eqref{alpha} as
\begin{eqnarray}
\lefteqn{\frac{j_{1}^2j_{2}^2}{(4j_{1}^2-n^2)((2\pi)^2j_{2}^2+n^2)}-\frac{k_{1}^2k_{2}^2}{(4k_{1}^2-n^2)((2\pi)^2k_{2}^2+n^2)}}\nonumber \\
& &  = \frac{j_{1}'^2j_{2}'^2}{(4j_{1}'^2-n^2)((2\pi)^2j_{2}'^2+n^2)}-\frac{k_{1}'^2k_{2}'^2}{(4k_{1}'^2-n^2)((2\pi)^2k_{2}'^2+n^2)},\label{eqAlpha}
\end{eqnarray} 
from which we obtain, up to reduction to common denominator,
\begin{eqnarray*}
j_{1}^2j_{2}^2(4k_{1}^2-n^2)(4j_{1}'^2-n^2)(4k_{1}'^2-n^2)((2\pi)^2k_{2}^2+n^2)((2\pi)^2j_{2}'^2+n^2)((2\pi)^2k_{2}'^2+n^2) & & \\
-k_{1}^2k_{2}^2(4j_{1}^2-n^2)(4j_{1}'^2-n^2)(4k_{1}'^2-n^2)((2\pi)^2j_{2}^2+n^2)((2\pi)^2j_{2}'^2+n^2)((2\pi)^2k_{2}'^2+n^2) & & \\
-j_{1}'^2j_{2}'^2(4j_{1}^2-n^2)(4k_{1}^2-n^2)(4k_{1}'^2-n^2)((2\pi)^2j_{2}^2+n^2)((2\pi)^2k_{2}^2+n^2)((2\pi)^2k_{2}'^2+n^2) & & \\
+k_{1}'^2k_{2}'^2(4j_{1}^2-n^2)(4j_{1}'^2-n^2)(4k_{1}^2-n^2)((2\pi)^2j_{2}^2+n^2)((2\pi)^2j_{2}'^2+n^2)((2\pi)^2k_{2}^2+n^2) & =& 0.
\end{eqnarray*}
We can rewrite the latter expression in the form 
 $P(2\pi)=0$ where $P$ is an integer polynomial 
of degree at most 6. 

Since $2\pi$ is a transcendental number, 
we necessarily have
 $P=0$. In particular, its leading coefficient  
vanishes, that is, 
\begin{align}
0 = &  j_{2}^2k_{2}^2j_{2}'^2k_{2}'^2
(j_{1}^2(4k_{1}^2-n^2)(4j_{1}'^2-n^2)(4k_{1}'^2-n^2) -k_{1}^2(4j_{1}^2-n^2)(4j_{1}'^2-n^2)(4k_{1}'^2-n^2)\nonumber \\
  & - j_{1}'^2(4j_{1}^2-n^2)(4k_{1}^2-n^2)(4k_{1}'^2-n^2)+k_{1}'^2(4j_{1}^2-n^2)(4k_{1}^2-n^2)(4j_{1}'^2-n^2))  .\label{a6}
\end{align}
A simple computation leads to
\begin{equation}\label{a6eq1}
(k_{1}^2-k_{1}'^2)(4j_{1}^2-n^2)(4j_{1}'^2-n^2)=(j_{1}^2-j_{1}'^2)(4k_{1}^2-n^2)(4k_{1}'^2-n^2).
\end{equation}

Recall that we are assuming ${\bf (k,j)\ne (k',j')}$ and that we want to prove that ${\bf (k,k')=(j,j')}$. 
Assume for now that 
\begin{equation}\label{absu}
k_1\ne k_1'.
\end{equation}
 According to \eqref{ind1} we also have $j_1\ne j_1'$. 
Equation \eqref{a6eq1}, moreover, yields
$$4(k_{1}^2k_{1}'^2-j_{1}^2j_{1}'^2)=n^2(k_1^2+k_1'^2-j_{1}^2-j_{1}'^2).$$
Using again \eqref{ind1} on both sides of the equality we get 
$$2(k_{1}^2(k_{1}^2-j_1^2+j_1'^2)-j_{1}^2j_{1}'^2)=n^2(k_{1}^2-j_{1}^2),$$
which implies 
$$(k_1^2-j_1^2)(n^2-2(k_1^2+j_1'^2))=0.$$
Since $n$ is odd, we necessarily have $n^2\ne  2(k_1^2+j_1'^2)$, which implies (jointly with \eqref{ind1})
$$k_1=j_1,\qquad k_1'=j_1'.$$
Equation \eqref{eqAlpha} becomes
\begin{equation}
\frac{k_{1}^2}{4k_{1}^2-n^2}\left(   \frac{j_{2}^2}{(2\pi)^2j_{2}^2+n^2}-\frac{k_{2}^2}{(2\pi)^2k_{2}^2+n^2}\right) = 
\frac{k_{1}'^2}{4k_{1}'^2-n^2}\left(\frac{j_{2}'^2}{(2\pi)^2j_{2}'^2+n^2}-\frac{k_{2}'^2}{(2\pi)^2k_{2}'^2+n^2}\right).\label{eqAlpha2}
\end{equation}

We are going to use several times the following technical result. 
\begin{lemma}\label{technique}
Let $\xi$ be a transcendental number, and take $a,b,c,d,\gamma\in \mathbb{N}$ and $\mu\in\mathbb{Q} \setminus\{0\}$.
If
\begin{equation}\label{fraz}
\frac{a}{a\xi+\gamma}-\frac{b}{b\xi+\gamma}=\mu\left( \frac{c}{c\xi+\gamma}-\frac{d}{d\xi+\gamma}\right)
\end{equation}
 then one of the properties holds true: (i)  $(a,c)=(b,d)$, (ii) $\mu=1$ and $(a,b)=(c,d)$, (iii) 
$\mu=-1$ and $(a,b)=(d,c)$.
 \end{lemma}
\begin{proof}
The proof consists simply in noticing that \eqref{fraz} is equivalent to the equality 
$$\frac{a}{a X+\gamma}-\frac{b}{bX+\gamma}=\mu\left( \frac{c}{cX+\gamma}-\frac{d}{dX+\gamma}\right)$$
between rational functions in the variable $X$ and in comparing their poles.  
\end{proof}

Applying the lemma to the identity  \eqref{eqAlpha2}, we get that either $(j_2,j_2')=(k_2,k_2')$, and hence ${\bf (k,k')=(j,j')}$ as desired,
or $\{j_2,k_2\}=\{j_2',k_2'\}$. In the latter case, moreover, \eqref{ind2} implies that $(j_2,k_2)=(j_2',k_2')$, which yields 
$$\frac{k_{1}^2}{4k_{1}^2-n^2}=\frac{k_{1}'^2}{4k_{1}'^2-n^2},$$
 since we are in case (ii) of Lemma~\ref{technique}.  
Since the map $x\mapsto x^2/(4x^2-n^2)$ is injective on $[1,+\infty)$ then $k_1=k_1'$, which contradicts \eqref{absu}.

Let now 
\begin{equation}\label{absu2}
k_1=k_1',\qquad k_2\ne k_2'.
\end{equation}
Identity \eqref{ind1} implies that $j_1=j_1'$ and equation \eqref{eqAlpha} simplifies to
\begin{equation}
\frac{j_{1}^2}{4j_{1}^2-n^2}\left(   \frac{j_{2}^2}{(2\pi)^2j_{2}^2+n^2}-\frac{j_{2}'^2}{(2\pi)^2j_{2}'^2+n^2}\right) = 
\frac{k_{1}^2}{4k_{1}^2-n^2}\left(\frac{k_{2}^2}{(2\pi)^2k_{2}^2+n^2}-\frac{k_{2}'^2}{(2\pi)^2k_{2}'^2+n^2}\right).\label{eqAlpha3}
\end{equation}
Let us apply again Lemma~\ref{technique}. Case (i) is ruled out by assumption \eqref{absu2}. Hence, $\{j_2,j_2'\}=\{k_2,k_2'\}$ and it follows  from \eqref{ind2}, using the same argument as before, that $(j_2,j_2')=(k_2,k_2')$ and $j_1=k_1$.
We conclude also in this second case that ${\bf (k,k')=(j,j')}$ and this concludes the proof of Proposition~\ref{noname}.
\end{proof}

\subsubsection{$\mathcal{P}^{0}_{2,\mathrm{BC}}$ is nonempty}\label{P2rho}

Let $\chi_n$ be defined as in the previous section (see equation \eqref{chin}).

\begin{proposition}\label{Sinfty}
If $n$ is even then $\chi_n\in \mathcal{P}^{0}_{2,\mathrm{BC}}$.
\end{proposition}
\begin{proof}
We use below the same parameterization on $\mathbb{N}^2$ of eigenpairs of the Laplace--Dirichlet operator as in Section~\ref{P1rho}.
Notice that the notion of connectedness chain introduced in Definition~\ref{d-chain} and \eqref{eq:S} and \eqref{eq:Srho} naturally extends to subsets of $(\mathbb{N}^2)^2$. 
 Then, $\chi_n$ is in $\mathcal{P}^{0}_{2,\mathrm{BC}}$ if and only if
$$\left\{({\bf j},{\bf k})\in (\mathbb{N}^2)^2\mid  \int_{\Omega} V^{\Id,\chi_n}_0(x)\phi_{\bf k}(x)\phi_{\bf j}(x)dx \neq 0\right\}$$
is a connectedness chain.

In order to prove that $S_0(\chi_n)=S(\Omega,\Gamma_D^g,\chi_n)$ 
is a connectedness chain, we are led to compute the quantities
$$
\int_\Omega V^{\Id,\chi_n}_0(x)\phi_{\bf j}(x)\phi_{\bf k}(x)dx=\frac{4}{L\pi}A_{n{\bf jk}}B_{n{\bf jk}},\qquad {\bf j}, {\bf k}\in (\mathbb{N}^2)^2,
$$
with
$$
A_{n{\bf jk}}=\int_0^\pi \sin (nx_1)\sin (j_1x_1)\sin (k_1x_1)dx_1
$$
and
$$
B_{n{\bf jk}}=\int_0^{L} \cosh (nx_2)\sin \left(\frac{j_2\pi x_2}{L}\right)\sin \left(\frac{k_2\pi x_2}{L}\right)dx_2.
$$
A tedious but straightforward computation proves that
$$
A_{n{\bf jk}}=\left\{\begin{aligned}
&0 && \textrm{if }j_1+k_1+n\textrm{ is even,}\\
\displaystyle &\frac{-4j_1k_1n}{(j_1+k_1-n)(j_1-k_1+n)(-j_1+k_1+n)(j_1+k_1+n)} && \textrm{otherwise,} 
\end{aligned}\right.
$$
whereas
$$
B_{n{\bf jk}}=\frac{2(-1)^{j_2+k_2} L^2n\pi^2j_2k_2\sinh(n  L)}{(\pi^2(j_2-k_2)^2+n^2)(\pi^2(j_2+k_2)^2+n^2)}.
$$
One immediately sees that the coefficients $B_{n{\bf jk}}$ cannot vanish. As for the coefficients $A_{n{\bf jk}}$, if $n$ is 
even then $A_{n{\bf jk}}$ vanishes if and only if $j_1$ and $k_1$ have the same parity.
Then  $S(\Omega,\Gamma_D^g,\chi_n)=\{({\bf j},{\bf k})\mid j_1+k_1 \mbox{ is odd}\}$ is a connectedness chain: indeed, given ${\bf j}$ and ${\bf k}$ in ${\mathbb{N}}^2$, either $j_1+k_1$ is odd, and then $({\bf j},{\bf k})\in S(\Omega,\Gamma_D^g,\chi_n)$, or $j_1+k_1$ is even and then $({\bf j},{\bf j}')$ and $({\bf j}',{\bf k})$ are in $S(\Omega,\Gamma_D^g,\chi_n)$ with ${\bf j}'=(j_1+1,j_2)$. 
\end{proof}

Notice that, conversely, if $n$ is odd 
then $A_{n{\bf jk}}$ vanishes if and only if $j_1+k_1$ is odd.
Hence, $S(\Omega,\Gamma_D^g,\chi_n)$ cannot couple ${\bf j}$ and ${\bf k}$ when $j_1+k_1$ is odd. 
Therefore, $\chi_n\not\in \mathcal{P}^{0}_{2,\mathrm{BC}}$ for $n$ odd.

\subsection{Proof of Theorem~\ref{grille-deformation}}\label{sec:grille-deformation}

According to Proposition \ref{p:P1P2} and Corollary~\ref{propagation}, we are left to prove that  $\mathcal{P}_1$ is nonempty. Indeed, we already showed in the previous section that $\mathcal{P}_{2,\mathrm{BC}}^0$ is nonempty, which implies that $\mathcal{P}_2$, which contains $\{(\Omega,\Gamma_D^g,\chi)\mid \chi\in \mathcal{P}_{2,\mathrm{BC}}^0\}$, is nonempty as well. 
We actually prove directly that $\mathcal{P}_1$ is residual, based on a general result proved in \cite{privat-sigalotti}.

\begin{lemma}\label{l:P1residual}
The set $\mathcal{P}_1$ is residual. 
\end{lemma}
\begin{proof}
Thanks to \cite[Theorem 2.3]{privat-sigalotti}, the lemma is proved if we show that 
for every $\ell\in\mathbb{N}$ and $q=(q_1,\dots,q_\ell)\in \mathbb{Q}^\ell \setminus \{0\}$
there exists $(Q(\Omega),Q(\Gamma_D^g),\chi)\in \mathcal{P}$ such that 
the first $\ell$ eigenvalues $\lambda_1,\dots,\lambda_\ell$ of the Dirichlet--Laplace operator on $Q(\Omega)$ are simple and 
$\sum_{j=1}^\ell q_j \lambda_j\ne 0$.

Fix $\ell\in\mathbb{N}$ and $q=(q_1,\dots,q_\ell)\in \mathbb{Q}^\ell \setminus \{0\}$. 
Let $\hat L>0$ be such that $\pi^2 \ell^2<\hat L^2$ and consider
$\hat \Omega=(0,\pi)\times (0,\hat L)$.
The choice of $\hat L$ is such that the $\ell$ smallest eigenvalues of $-\Delta$ on $\hat \Omega$ with Dirichlet boundary conditions are $\lambda_j=1+j^2 \pi^2/\hat L^2$, which are  simple and whose corresponding eigenfunctions 
are (up to normalization) 
$$\phi_j(x_1,x_2)=\frac{2\sin(x_1)\sin(j \pi x_2/\hat L)}{\sqrt{\pi \hat L}}.$$

Let $X$ be a $\mathcal{C}^1$ vector field on $\R^2$ with compact support intersecting $\{0\}\times(0,\hat L)$ but not any other side of $\hat \Omega$. For $t_0>0$ small enough and $t\in (-t_0,t_0)$, 
$I+tX$ is a diffeomorphism between $\hat\Omega$ and its image, which we will denote by $\hat\Omega_t$.

Denote by $(\lambda_j(t))_{j\in\mathbb{N}}$ the spectrum 
of the  Laplace--Dirichlet operator on 
$\hat\Omega_t$. According to Rellich's theorem (see \cite{kato,rellich}), each function $\lambda_j(\cdot)$ can be chosen to be analytic on $(-t_0,t_0)$. Moreover, up to reducing $t_0$, we can assume that $\lambda_1(t),\dots,\lambda_\ell(t)$ are simple for $t\in(-t_0,t_0)$. 

It is well known that
$$
\dot\lambda_j(0)=-\int_{\partial \hat \Omega}\left(\frac{\partial \phi_j}{\partial \nu }\right)^2(X\cdot \nu)
$$
for every $j$ such that $\lambda_j(0)$ is simple
(see, for instance, \cite{Henrot-Pierre}).
Notice that 
$$\left(\frac{\partial \phi_j}{\partial \nu }\right)^2=\frac{4\sin^2(j \pi x_2/\hat L)}{\pi \hat L}$$
on $\{0\}\times(0,\hat L)$ for $j=1,\dots,\ell$.

Henceforth, since $x_2\mapsto \sin^2(j\pi x_2/\hat L)$, $j=1,\dots,\ell$, are linearly independent functions on $(0,\hat L)$ (as it  follows from the trigonometric formula $\sin^2 \theta=1-\cos(2\theta)/2$ and by injectivity of Fourier series), then 
we can choose the vector field  $X$ in such a way that 
$$\sum_{j=1}^\ell q_j \dot\lambda_j(0)\ne 0.$$  
Hence, there exists $t\in(-t_0,t_0)$ such that 
$\sum_{j=1}^\ell q_j \lambda_j(t)\ne 0$ and the lemma is proved taking $Q=\Id+t X$.
\end{proof}

\section{Generalizations}\label{s-generalizations}

In this section we provide some generalizations of the results obtained in Theorems~\ref{grille-tangent} and \ref{grille-deformation}. In Section~\ref{s:partial} we consider gates which do not cover the entire upper side of the rectangle $\Omega$. In Section~\ref{s:nonlinear} we include some physically motivated nonlinear correction to the coupling term between the Poisson and the Schr\"odinger equation.

\subsection{Partial gate with linear coupling}\label{s:partial}

The model that we consider here is the following, 
\begin{equation}\label{schro1}
\left\{\begin{aligned}
i\partial_t\psi (t,x)&= -\Delta \psi (t,x)+V(t,x)\psi(t,x),\quad  && (t,x)\in  \mathbb{R}_+\times\Omega, \\
-\Delta V(t,x) &= 0, && (t,x)\in  \mathbb{R}_+\times\Omega, \\
\psi (t,x)&=0, && (t,x)\in \mathbb{R}_+\times \partial \Omega,\\
V (t,x)&=V_g(t)\chi(x), && (t,x)\in \mathbb{R}_+\times \Gamma_D^g,\\
V (t,x)&=0, && (t,x)\in \mathbb{R}_+\times \Gamma_D^s\cup \Gamma_D^d,\\
\frac{\partial V}{\partial \nu }(t,x)&=0, && (t,x)\in \mathbb{R}_+\times \Gamma_N.
\end{aligned}\right.
\end{equation}
The set $\Omega$ still denotes the rectangle $(0,\pi)\times (0,L)$, $L>0$.  
The gate $\Gamma_D^g$ is now reduced to a compactly contained subinterval of $[0,\pi]\times \{L\}$, while 
$\Gamma_N=\Gamma^1_N\cup \Gamma^2_N \cup \Gamma_N^3$ is now the union of three connected components, as illustrated in Figure~\ref{figTransistor2}.

\vspace{1cm}
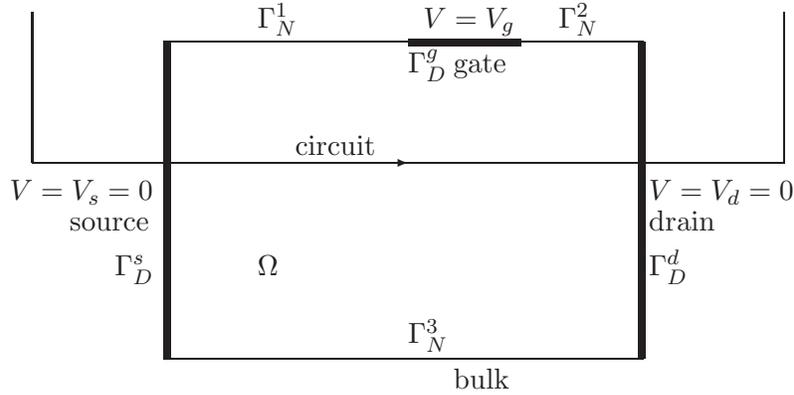
\begin{figure}[!ht]
 \begin{center}
\fbox{
\begin{picture}(6,4)
\put(1,1){$\Omega$}
\put(3.2,4.3){$V=V_g$}
\put(3.6,3.7){gate}
\put(-2.3,2){$V=V_s=0$}
\put(-1.5,1.6){source}
\put(6.2,2){$V=V_d=0$}
\put(6.2,1.6){drain}
\put(1.5,2.6){circuit}
\put(3.6,-0.5){bulk}
\put(1,4.3){$\Gamma_N^1$}
\put(5,4.3){$\Gamma_N^2$}
\put(3.0,3.7){$\Gamma_D^g$}
\put(-0.9,1){$\Gamma_D^s$}
\put(6.2,1){$\Gamma_D^d$}
\put(3,0.1){$\Gamma_N^3$}
\put(-2,2.5){\line(0,1){2}}
\put(-2,2.5){\vector(1,0){5}}
\put(3,2.5){\line(1,0){5}}
\put(8,2.5){\line(0,1){2}}
\linethickness{0.9mm}
\put(-0.2,-0.1){\line(0,1){4.2}}  
\put(6.1,-0.1){\line(0,1){4.2}}  
\put(3,4.1){\line(1,0){1.5}}
\end{picture}}
\end{center}
\caption{Representation of the transistor with partial gate}\label{figTransistor2}
\end{figure}

\bigskip

As in the previous sections, we can consider a deformation of $\Omega$ by introducing a transformation $Q\in \mathrm{Diff}^1_0$, with $\chi\in \mathcal{C}^1(Q(\Gamma_D^g))$. 
Similarly to what is done in Section~\ref{enonces} and with a slight abuse of notations, we denote by $\mathcal{P}$ the  class of corresponding problems, identified with  
\begin{equation}\label{def-Pbis}
\mathcal{P}=\{(Q(\Omega),Q(\Gamma_D^g),\chi)\mid 
(Q(\Omega),Q(\Gamma_D^g))\in \Sigma\quad \mbox{and}\quad
\chi\in \mathcal{C}^1(Q(\Gamma^g_D))\},
\end{equation}
where
$$\Sigma=\{{(Q(\Omega),Q(\Gamma_D^g))}\mid Q\in \mathrm{Diff}^1_0\}.$$ 

We obtain the following result. 
\begin{theorem}\label{grille-deformation-bis}
For a generic element of $\mathcal{P}$, the control problem \eqref{schro1} is approximately controllable.
\end{theorem}
\begin{proof}
The proof consists in an adaptation of 
the one of Theorem~\ref{grille-deformation}. We denote by $\mathcal{P}_1$ and $\mathcal{P}_2$ the sets defined in analogy to what done in Section~\ref{s:prelim}. The same argument as in Proposition~\ref{p:P1P2} allows us to prove the theorem by showing that 
 $\mathcal{P}_1$ and $\mathcal{P}_2$ are residual. 
 
 Notice that the condition defining the set $\mathcal{P}_1$ actually depends only on $Q(\Omega)$, and not on $Q(\Gamma_D^g)$ and $\chi$. Hence, as proved in Lemma~\ref{l:P1residual}, $\mathcal{P}_1$ is residual.

Let us focus on the set $\mathcal{P}_2$. It is crucial for our argument to notice that 
the analyticity of $V_0^{Q,\chi}$ with respect to $Q$ and $\chi$ still holds in the case of partial gates, as it can 
be seen by a straightforward adaptation of 
Proposition~\ref{p:analytic}. As a consequence, as it was done in 
Corollary~\ref{propagation}, it is sufficient to prove that the set $\mathcal{P}_2$ is nonempty. 
For that purpose we proceed by defining a suitable subclass of $\mathcal{P}$ in which we are able to prove the density of 
$\mathcal{P}_2$.

Indeed, consider 
$\widetilde  L>0$ such that $\widetilde L^2\not\in \pi^2 \mathbb{Q}$ and define $\widetilde \Omega=(0,\pi)\times (0,\widetilde L)$. 
Let us introduce the subclass 
$\widetilde{\mathcal{P}}$ of $\mathcal{P}$
defined by
$$\widetilde{\mathcal{P}}=\{(Q(\Omega),Q(\Gamma_D^g),\chi)\in\mathcal{P}\mid 
Q(\Omega)=\widetilde  \Omega,\ Q(\Gamma_D^g)\subset [0,\pi]\times\{\widetilde  L\}\}.$$
Denote by $(\phi_{\bf j})_{{\bf j}\in\mathbb{N}^2}$ an $L^2$-orthonormal basis for the Laplace--Dirichlet operator on $\widetilde  \Omega$. 

Let $n\in \mathbb{N}$ be even, $\chi_n$ be defined as in \eqref{chin} (see Proposition~\ref{Sinfty}) and let 
$$
\overline S=S(Q(\Omega),Q(\Gamma_D^g),\chi_{n}).
$$
The intersection of $\mathcal{P}_2$ with $\widetilde{\mathcal{P}}$ 
contains in particular those elements $(Q(\Omega),Q(\Gamma_D^g),\chi)\in\widetilde{\mathcal{P}}$ such that
$\overline S\subset S(Q(\Omega),Q(\Gamma_D^g),\chi)$, i.e., 
$$\mathcal{P}_2\cap \widetilde{\mathcal{P}}\supset \bigcap_{({\bf j, k})\in \overline S}\mathcal{O}_{\bf jk}$$
where, for every ${\bf j, k}\in \mathbb{N}^2$,  
$$\mathcal{O}_{\bf jk}=\{(Q(\Omega),Q(\Gamma_D^g),\chi)\in\widetilde{\mathcal{P}}\mid \int_{\widetilde \Om} V_0^{Q,\chi}\phi_{\bf j}\phi_{\bf k}\ne 0\}.$$

Clearly, each $\mathcal{O}_{\bf jk}$ is open in $\widetilde{\mathcal{P}}$. 
The proof of the theorem is concluded by showing that 
$\mathcal{O}_{\bf jk}$ is dense for every $({\bf j, k})\in \overline S$.
Actually, we just need to prove that for every $({\bf j, k})\in \overline S$ there exists an element $(Q^{\bf j, k}(\Omega),Q^{\bf j, k}(\Gamma_{p,D}^g),\chi^{\bf j, k})$ in  $\mathcal{O}_{\bf jk}$: indeed, any other  element of $\widetilde{\mathcal{P}}$ can be connected to $(Q^{\bf j, k}(\Omega),Q^{\bf j, k}(\Gamma_{p,D}^g),\chi^{\bf j, k})$ by an analytic path within $\widetilde{\mathcal{P}}$, along which 
$V_0^{Q,\chi}$ varies analytically (while $\phi_{\bf j}$ and $\phi_{\bf k}$ do not vary at all). In particular, almost every element of the path is in 
$\mathcal{O}_{\bf jk}$, whence the density of $\mathcal{O}_{\bf jk}$ in $\widetilde{\mathcal{P}}$.

Let us introduce a sequence  
 $(\Gamma_{p,D}^g)_{p\in\mathbb{N}}$
of segments 
included in $(0,\pi)\times \{\widetilde  L\}$
increasing for the inclusion and such that
$$
\bigcup_{p=1}^{+\infty}\Gamma_{p,D}^g=(0,\pi)\times \{\widetilde  L\}.
$$
For every $p\in \mathbb{N}$ let 
$Q_p\in \mathrm{Diff}^1_0$ be such that $Q_p(\Om)=\widetilde  \Om$ and $Q_p(\Gamma_D^g)=\Gamma_{p,D}^g$, and 
$$\eta_p= \chi_n\Big|_{\Gamma_{p,D}^g}\in \mathcal{C}^1(\Gamma_{p,D}^g).$$

The following continuity result holds true and concludes the proof of the theorem.
\begin{lemma}\label{propCont}
Define $V_0^p=V_0^{Q_p,\eta_p}$.  
The sequence $(V^{p}_0)_{p\in\mathbb{N}}$ converges strongly in $H^1(\widetilde  \Omega)$ to $V^{\Id,\chi_n}_0$ as $p\to+\infty$. 
\end{lemma}

By a slight notational abuse we denote by $\chi_n$ its extension   on $\widetilde  \Om$ satisfying $\chi_n(x_1,x_2)=\chi_n(x_1,\widetilde  L)$ for every $(x_1,x_2)\in \widetilde  \Om$. 
Let us introduce the lift $W_p=V^{p}_0-\chi_n$. Thus, $W_p$ is the solution of the following partial differential equation
\begin{equation}\label{defu}
\left\{\begin{aligned}
-\Delta W_p(x) &=n^2 \chi_n(x),\quad && x\in (0,\pi)\times (0,\widetilde  L),\\
W_p(x)&=0, && x\in \Gamma_{p,D}^g\cup  (\{0,\pi\}\times [0,\widetilde  L]),\\
\frac{\partial W_p}{\partial \nu }(x)&=0, && x\in ([0,\pi]\times \{0\}) \cup (([0,\pi]\times \{\widetilde  L\})\backslash \Gamma_{p,D}^g),
\end{aligned}
\right.
\end{equation}
whose  variational formulation is written as follows: find $W_p$ in 
$$\mathcal{V}_p=\left\{v\in H^1(\Omega)\mid v=0\mbox{ on }\Gamma_{p,D}^g\cup (\{0,\pi\}\times [0,\widetilde  L])\right\}$$ such that for every $v\in\mathcal{V}_p$, one has
\begin{equation}\label{FVup}
\int_{\widetilde  \Omega} \nabla W_p(x)\cdot \nabla v(x)dx=n^2\int_{\widetilde  \Omega} v(x)\chi_n(x)dx.
\end{equation}
By definition, each $V^{p}_0$ is harmonic and  reaches its maximal and minimal values on the boundary of $\widetilde  \Omega$ at some points where the normal derivative of $V^{p}_0$ does not vanish, as it follows from the Hopf maximum principle. 
Thus, $\displaystyle \Vert V^{p}_0\Vert_{L^\infty(\widetilde  \Omega)}\leq \max_{x_1\in [0,\pi]}|\chi_n(x)|\leq \cosh (n\widetilde  L)$. As a consequence, the sequences $(V^{p}_0)_{p\in\mathbb{N}}$ and $(W_p)_{p\in\mathbb{N}}$ are uniformly bounded (with respect to $p$) in $L^2(\widetilde  \Omega)$. Taking now $v=W_p$ in \eqref{FVup} yields
$$
\Vert \nabla W_p\Vert_{L^2(\widetilde  \Omega)}^2\leq n^2 \Vert \chi_n\Vert_{L^2(\widetilde  \Omega)}\Vert W_p\Vert_{L^2(\widetilde  \Omega)}.
$$
The sequence $(W_p)_{p\in\mathbb{N}}$ is thus bounded in $H^1(\widetilde  \Omega)$ and, from Rellich compactess embedding theorem, converges up to a subsequence weakly in $H^1(\widetilde  \Omega)$ and strongly in $L^2(\widetilde  \Omega)$ to some $W_\infty\in H^1(\widetilde  \Omega)$. In the sequel, we will still denote by $(W_p)_{p\in\mathbb{N}}$ the considered subsequence. Taking tests functions $v$ in $\mathcal{C}^\infty(\overline{\widetilde  \Omega})$ with compact support in \eqref{FVup} yields immediately that $W_\infty$ satisfies
$$
-\Delta W_\infty=n^2 \chi_n
$$
in distributional sense. By compactness of the trace operator, one has necessarily $W_\infty=0$ on $\{0,\pi\}\times [0,\widetilde  L]$. 
Since the sequence $(\Gamma^g_{p,D})_{p\in\mathbb{N}}$  is increasing for the inclusion and converges to $(0,\pi)\times \{\widetilde L\}$, one sees that for any compact $K\subset (0,\pi)\times \{\widetilde L\}$ there exists $p_0$ such that  $W_p=0$ on $K$ for every $p\geq p_0$. Thus, one yields $W_{\infty}=0$ on $(0,\pi)\times \{\widetilde L\}$.
Finally, since $\mathcal{V}_p$ is increasing with respect to $p$, it is obvious that for every $p\in\mathbb{N}$ and $v\in\mathcal{V}_p$, $W_{\infty}$ satisfies
\begin{equation}\label{FVustar}
\int_{\widetilde \Omega} \nabla W_{\infty}(x)\cdot \nabla v(x)dx=n^2\int_{\widetilde \Omega} v(x)\chi_{\infty}(x)dx.
\end{equation}
Introduce $\mathcal{V}_\infty = \bigcup_{p=0}^{+\infty}\mathcal{V}_P$, that is,
$$
\mathcal{V}_\infty=\left\{v\in H^1(\Omega)\mid v=0\mbox{ on }((0,\pi)\times \{\widetilde L\})\cup   (\{0,\pi\}\times [0,\widetilde  L])\right\}.
$$
It is clear that $W_{\infty}$ satisfies \eqref{FVustar} for every $v\in\mathcal{V}_\infty$. By taking $v=W_p$ in \eqref{FVup} and since $(W_p)_{p\in\mathbb{N}}$ converges strongly in $L^2(\widetilde \Omega)$ to $W_{\infty}$, it follows that $\Vert W_p\Vert_{H^1(\widetilde \Omega)}$ converges to $\Vert W_{\infty}\Vert_{H^1(\widetilde \Omega)}$ as $p\to +\infty$. Since $(W_p)_{p\in\mathbb{N}}$ also converges weakly in $H^1(\widetilde \Omega)$ to $W_{\infty}$, we deduce that this convergence is in fact strong in $H^1(\widetilde \Omega)$, whence the result.
\end{proof}

\subsection{Nonlinear coupling}\label{s:nonlinear}

In this section, we show how the approximate controllability results proved in the previous sections can be applied to obtain some suitable controllability property for a nonlinear system. We now take into account selfconsistent electrostatic interactions between electrons in the Poisson equation. For simplicity, we only consider the case where the gate covers the entire upper side of the domain $\Omega$.

We consider here the following Schr\"odinger--Poisson system,
\begin{equation}\label{schropois}
\left\{\begin{aligned}
i\partial_t\psi (t,x)&= -\Delta \psi (t,x)+V(t,x)\psi(t,x),\quad && (t,x)\in  \mathbb{R}_+\times\Omega, \\
-\Delta V(t,x) &= \alpha|\psi(t,x)|^2, && (t,x)\in  \mathbb{R}_+\times\Omega, \\
\psi (t,x)&=0, && (t,x)\in \mathbb{R}_+\times \partial \Omega,\\
V (t,x)&=V_g(t)\chi(x), && (t,x)\in \mathbb{R}_+\times \Gamma_D^g,\\
V (t,x)&=0, && (t,x)\in \mathbb{R}_+\times \Gamma_D^s\cup \Gamma_D^d,\\
\frac{\partial V}{\partial \nu }(t,x)&=0, && (t,x)\in \mathbb{R}_+\times \Gamma_N.
\end{aligned}\right.
\end{equation}
Here, $\alpha>0$  denotes a dimensionless parameter that quantifies the strength of nonlinear effects; $1/\sqrt{\alpha}$ is the so-called \emph{scaled Debye length}. The domain $\Omega$ is the rectangle $(0,\pi)\times (0,L)$ in the configuration of Figure \ref{figTransistor}: the gate is the entire segment $\Gamma_D^g=[0,\pi]\times \{L\}$ and the Neumann boundary is $\Gamma_N=[0,\pi]\times \{0\}$.

In order to exploit elliptic regularity properties, we consider here smoother perturbation parameters than in previous sections, taking the diffeomorphism $Q$ in $\mathrm{Diff}^2_0$ (the class of $\mathcal{C}^2$ orientation-preserving diffeomorphism of $\R^2$)
and $\chi\in \mathcal{C}^1_0(Q(\Gamma_D^g))$.

It is convenient to split the potential into the sum of the control potential and the nonlinear potential as $V(t,x)=V_g(t)V_0(x)+W_\psi(t,x)$. The resulting equation in the deformed domain can be written as
\begin{equation}\label{schropois2}
i\partial_t\psi = -\Delta \psi +V_g(t)V_0\psi+W_\psi\psi,\qquad \mbox{in }\mathbb{R}_+\times Q(\Omega),
\end{equation}
where $V_0$ and $W_\psi$ are the solutions of
\begin{equation}\label{defV0-nonlinear}
\left\{\begin{aligned}
-\Delta V_0(x) &=0, && x\in Q(\Omega),\\
V_0(x)&=\chi(x), && x\in Q(\Gamma_D^g),\\
V_0 (x)&=0, && x\in Q(\Gamma_D^s\cup \Gamma_D^d),\\
\frac{\partial V_0}{\partial \nu }(x)&=0, && x\in Q(\Gamma_N),
\end{aligned}
\right.
\end{equation}
and
\begin{equation}\label{nonlinear-s}
\left\{\begin{aligned}
-\Delta W_\psi(t,x) &=\alpha|\psi(t,x))|^2, && (t,x)\in \mathbb{R}_+\times Q(\Omega), \\
W_\psi(t,x)&=0, && (t,x)\in \mathbb{R}_+\times Q(\Gamma_D^g\cup \Gamma_D^s \cup \Gamma_D^d),\\
\frac{\partial W_\psi}{\partial \nu }(t,x)&=0, && (t,x)\in \mathbb{R}_+\times  Q(\Gamma_N). 
\end{aligned}
\right.
\end{equation}

Before stating our approximate controllability result for \eqref{schropois2}, we address the question of well-posedness of this Cauchy problem. Two kinds of results are available for Schr\"odinger--Poisson systems, see \cite{cazenave}. In the whole-space case $\Omega=\R^2$, Strichartz estimates enable to benefit from the dispersive and smoothing properties of the Schr\"odinger group and construct a unique global $L^2$ solution to the problem \cite{GV,HO,castella}. In a general domain $\Omega$, for more regular initial data in $H^1$ or $H^2$, the analysis is simpler and the proof of global well-posedness can rely on energy estimates, see \cite{arnold,BM,IZL,MRS}. 

However, none of these results apply to our situation, which requires a specific study. Indeed, dealing with a problem set on general bounded domains $Q(\Omega)$, our analysis cannot rely on Strichartz estimate and we have to assume that the Cauchy data are more regular than $L^2$, for instance that they belong to the energy space $H^1$. In this case, the proof of local in time existence and uniqueness of a solution to \eqref{schropois2} is not a difficult task and the main issue is the question of {\em global} existence. As we said above, the proof of global existence usually relies on an energy estimate for \eqref{schropois2}, \eqref{defV0-nonlinear}, \eqref{nonlinear-s}. When the applied potential $V_g(t)$ is differentiable, this estimate can be obtained by multiplying \eqref{schropois2} by $\partial_t \overline \psi$ and integrating on $[0,t]\times Q(\Omega)$, and reads
\begin{align*}
&\|\nabla \psi(t)\|_{L^2(Q(\Omega))}^2+\frac{1}{2\alpha}\|\nabla W_\psi(t)\|^2+V_g(t)\int_{Q(\Omega)}V_0(x)|\psi(t,x)|^2\, dx\\
&\qquad \qquad \quad =\|\nabla \psi_0\|_{L^2(Q(\Omega))}^2+\frac{1}{2\alpha}\|\nabla W_{\psi_0}\|^2+V_g(0)\int_{Q(\Omega)}V_0(x)|\psi_0(x)|^2\, dx\\
&\qquad \qquad \qquad +\int_0^t \int_{Q(\Omega)}\partial_tV_g(s)V_0(x)|\psi(s,x)|^2\,ds dx.
\end{align*}

For completeness, we consider in the following nonsmooth control functions $V_g\in L^\infty([0,T],[0,\delta])$, where $\delta>0$ is given: for instance, $V_g$ can be piecewise constant. In the general case, we follow another path to prove that the energy of the system --\,say the $H^1$ norm of $\psi$\,-- remains bounded on any $[0,T]$, independently of the derivative of the control. We state this result in the following proposition, whose proof is based on a Br\'ezis--Gallouet type argument \cite{brezis-gallouet}.

\begin{proposition}
\label{propCauchy}
Let $T>0$, let $Q\in \mathrm{Diff}^2_0$, let $\chi\in {\mathcal C}^1_0(Q(\Gamma^g_D))$, and let $V_g\in L^\infty([0,T],[0,\delta])$. Then, for every $\psi_0\in H^1_0(Q(\Omega))$,  the system \eqref{schropois2}, \eqref{defV0-nonlinear}, \eqref{nonlinear-s} admits a unique mild solution $\psi\in \mathcal{C}^0([0,T],H^1_0(Q(\Omega)))$ and there exists $c>0$ such that, for all $t\in [0,T]$,
\begin{equation}
\label{apriori}
\|\psi(t,\cdot)\|_{H^1}\leq \exp\left(ce^{ct}\right).
\end{equation}
The constant $c$ only depends on $\delta$, $\alpha_0$, $Q$, $\|\psi_0\|_{H^1}$ and $\|\chi\|_{\mathcal C^1}$.
\end{proposition}
\begin{proof}
Let us first prove the local well-posedness of the Cauchy problem in $H^1_0(Q(\Omega))$. For all $T_0>0$, we set $X_{T_0}={\mathcal C}^0([0,T_0],H^1_0(Q(\Omega)))$ with the norm
$$\|u\|_{X_{T_0}}=\max_{t\in[0,T_0]}\|\nabla u(t)\|_{L^2}.$$
Denoting by $(e^{i\tau \Delta})_{\tau\in \R}$ the group of unitary  transformations generated by the operator $i\Delta$ with Dirichlet boundary conditions, a mild solution $\psi\in X_{T_0}$ of \eqref{schropois2} satisfies
 \begin{align}
 \label{mild}
 \psi(t,\cdot)=e^{-i t \Delta}\psi_0(\cdot)+\int_0^t e^{-i (t-s) \Delta}(V_g(s) V_0(\cdot)+W_\psi(s,\cdot))\psi(s,\cdot)ds,
 \end{align}
where $V_0$ and $W_\psi$ are defined by \eqref{defV0-nonlinear} and  \eqref{nonlinear-s}, and can be characterized as a fixed-point of the mapping $S\,:\,X_{T_0}\to X_{T_0}$ given by
$$
S(\psi)=e^{-i t \Delta}\psi_0(\cdot)+\int_0^t e^{-i (t-s) \Delta}(V_g(s) V_0(\cdot)+W_\psi(s,\cdot))\psi(s,\cdot)ds.
$$
Let $R> \|\nabla \psi_0\|_{L^2}$ be fixed and define
$${\mathcal B}_R=\{u\in X_{T_0}\,:\,\|u\|_{X_{T_0}}\leq R\}.$$
We will prove that, for $T_0$ small enough, $S$ is a contraction mapping on ${\mathcal B}_R$. 

By elliptic regularity, since the function $\chi$ belongs to $\mathcal C^1$ and vanishes at the boundary of the grid, the fixed potential $V_0$ which solves \eqref{defV0-nonlinear} belongs (at least) to $H^{3/2}(Q(\Omega))$. 
Denoting in the following by $C$ any positive constant depending only on the domain $Q(\Omega)$, the Sobolev embeddings $H^{3/2}\hookrightarrow W^{1,4}$, $H^1\hookrightarrow L^4$, $H^{3/2}\hookrightarrow L^\infty$ and the Poincar\'e inequality yield
\begin{align}
\|\nabla (V_0\psi)\|_{L^2}&\leq \|(\nabla V_0)\psi\|_{L^2}+\|V_0 \nabla \psi\|_{L^2}\leq \|\nabla V_0\|_{L^4}\|\psi\|_{L^4}+\|V_0\|_{L^\infty} \|\nabla\psi\|_{L^2}\nonumber\\&\leq C\|V_0\|_{H^{3/2}} \|\nabla\psi\|_{L^2}\,.\label{V0psi}
\end{align}
By elliptic regularity and Sobolev embeddings, we have for all $\psi,\,\widetilde \psi \in H^1_0(\Omega)$
\begin{align}
\|W_\psi-W_{\widetilde \psi}\|_{H^2}&\leq C\alpha \||\psi|^2-|\widetilde \psi|^2\|_{L^2}\leq C\alpha (\|\psi\|_{L^4}+\| \widetilde \psi\|_{L^4})\|\psi-\widetilde \psi\|_{L^4}\nonumber\\
&\leq C\alpha (\|\nabla \psi\|_{L^2}+\|\nabla \widetilde \psi\|_{L^2})\|\nabla (\psi-\widetilde \psi)\|_{L^2},\label{ellip}
\end{align}
so, proceeding as for \eqref{V0psi}, we get
\begin{align*}
\|\nabla (W_\psi \psi-W_{\widetilde \psi} \widetilde \psi)\|_{L^2}\leq C \alpha (\|\nabla \psi\|_{L^2}^2+\|\nabla \widetilde \psi\|_{L^2}^2)\|\nabla(\psi-\widetilde \psi)\|_{L^2}\,.
\end{align*}
Finally, using that $e^{-i t \Delta}$ is unitary on $H^1_0(Q(\Omega))$, we obtain, for all $\psi,\,\widetilde \psi \in {\mathcal B}_R$
\begin{align*}
\|S(\psi))\|_{X_{T_0}}&\leq \|\nabla\psi_0\|_{L^2}+\int_0^{T_0}|V_g(s)|\left(\|\nabla(V_0\psi)(s)\|_{L^2}+\|\nabla (W_\psi \psi)(s)\|_{L^2}\right)ds\\
&\leq \|\nabla\psi_0\|_{L^2}+CT_0\left(\|V_g\|_{L^\infty}\|V_0\|_{H^{3/2}}R+\alpha R^3\right)
\end{align*}
and
\begin{align*}
\|S(\psi)-S(\widetilde \psi)\|_{X_{T_0}}&\leq \int_0^{T_0}|V_g(s)|\left(\|\nabla(V_0(\psi-\widetilde \psi))(s)\|_{L^2}+\|\nabla (W_\psi \psi-W_{\widetilde \psi} \widetilde \psi)(s)\|_{L^2}\right)ds\\
&\leq CT_0\left(\|V_g\|_{L^\infty}\|V_0\|_{H^{3/2}}+2\alpha R^2\right)\|\psi-\widetilde \psi\|_{X_{T_0}}.
\end{align*}
Hence, it is clear that, since $ \|\nabla \psi_0\|_{L^2}<R$, choosing $T_0$ small enough ensures $S(\psi)\in {\mathcal B}_R$ and $\|S(\psi)-S(\widetilde \psi)\|_{X_{T_0}}<q\|\psi-\widetilde \psi\|_{X_{T_0}}$ with $q<1$. Then, the Banach fixed-point theorem implies the existence of a unique mild solution to \eqref{schropois2} on the time interval $[0,T_0]$. Furthermore, if the \emph{a priori} estimate \eqref{apriori} is proved, then by a standard continuation argument, the existence interval can be taken equal to $[0,T]$, which means that the solution is in fact global in time.

\bigskip
Let us now prove the crucial estimate \eqref{apriori}. We first recall that the $L^2$ norm of $\psi$ is an invariant of \eqref{schropois2}: for all $t\geq 0$, one has $\|\psi(t)\|_{L^2}=\|\psi_0\|_{L^2}$. To estimate the $H^1_0$ norm of $\psi$, we come back to \eqref{mild} which yields
\begin{align}
\|\nabla \psi(t)\|_{L^2}&\leq \|\nabla \psi_0\|_{L^2}+\|V_g\|_{L^\infty }\int_0^t \|\nabla (V_0 \psi)(s)\|_{L^2}ds+\int_0^t \|\nabla (W_\psi \psi)(s)\|_{L^2}ds\nonumber\\
&\leq \|\nabla \psi_0\|_{L^2}+C\|V_g\|_{L^\infty}\|V_0\|_{H^{3/2}}\int_0^t \|\nabla \psi(s)\|_{L^2}ds+\int_0^t \|\nabla (W_\psi \psi)(s)\|_{L^2}ds,
\label{est1}
\end{align}
where we used \eqref{V0psi}. We thus need to estimate the product
\begin{equation}
\label{prod}
\|\nabla (W_\psi \psi)\|_{L^2}\leq \|(\nabla W_\psi) \psi\|_{L^2}+\|W_\psi \nabla \psi\|_{L^2}\leq \|\nabla W_\psi\|_{L^4}\|\psi\|_{L^4}+\|W_\psi\|_{L^\infty}\|\nabla \psi\|_{L^2}\,.
\end{equation}
For the first term, we use elliptic regularity and Sobolev embedding,
$$\|\nabla W_\psi\|_{L^4}\leq C\|W_\psi\|_{W^{2,4/3}}\leq C\alpha \||\psi|^2\|_{L^{4/3}}= C\alpha \|\psi\|_{L^{8/3}}^2.$$
Next, we recall the following two Gagliardo--Nirenberg inequalities: for all $\psi\in H^1_0(Q(\Omega))$, one has
$$\|\psi\|_{L^{8/3}}^2\leq \|\nabla \psi\|_{L^2}^{1/2}\|\psi\|_{L^2}^{3/2},\qquad 
\|\psi\|_{L^{4}}\leq \|\nabla \psi\|_{L^2}^{1/2}\|\psi\|_{L^2}^{1/2}.$$
Hence, the first term in the right hand side of \eqref{prod} can be bounded linearly in $\|\nabla \psi\|_{L^2}$ as
\begin{equation}
\label{first}
\|\nabla W_\psi\|_{L^4}\|\psi\|_{L^4}\leq C\alpha \|\psi\|_{L^2}^2 \|\nabla \psi\|_{L^2}=C\alpha \|\psi_0\|_{L^2}^2 \|\nabla \psi\|_{L^2}.
\end{equation}

The main source of concern is the second term in the right hand side of \eqref{prod}. Indeed, $\|W_\psi\|_{L^\infty}$ cannot be bounded by a quantity which only depends on the $L^2$ norm of $\psi$ (an $L^1$ right-hand side in the elliptic equation \eqref{nonlinear-s} does not produce an $L^\infty$ potential), so this term will necessarily lead to a super-linear estimate in $\|\nabla \psi\|_{L^2}$. 

A key inequality in the proof will be the following one, proved by Br\'ezis and Gallouet in \cite{brezis-gallouet}. There exists a constant $C>0$ such that, for all $u\in H^2(\Omega)$, one has
\begin{equation}
\label{BG}
\|u\|_{L^\infty}\leq C(1+\|u\|_{H^1}\sqrt{\log(1+\|u\|_{H^2})}).
\end{equation}

Let us multiply the first equation of \eqref{nonlinear-s} by $W_\psi$ and integrate on $Q(\Omega)$. After an integration by parts, it comes
\begin{equation}
\label{c}\|\nabla W_\psi\|^2_{L^2}=\alpha \int_{Q(\Omega)}W_\psi |\psi|^2 dx\leq \alpha \|W_\psi\|_{L^\infty}\|\psi_0\|_{L^2}^2.
\end{equation}
Therefore, from the Poincar\'e inequality, from \eqref{BG} and \eqref{c}, we deduce
$$\|W_\psi\|_{L^\infty}^2\leq C(1+\alpha \|W_\psi\|_{L^\infty}\|\psi_0\|_{L^2}^2\log(1+\|W_\psi \|_{H^2})),$$
from which we get
$$\|W_\psi\|_{L^\infty}\leq C(1+\alpha \|\psi_0\|_{L^2}^2\log(1+\|W_\psi \|_{H^2})).$$
Next, using \eqref{ellip} with $\widetilde\psi =0$, we obtain
\begin{equation}
\label{estiW}
\|W_\psi\|_{L^\infty}\leq C(1+\alpha \|\psi_0\|_{L^2}^2\log(1+\sqrt{\alpha}\|\nabla \psi \|_{L^2})).
\end{equation}

Finally, gathering \eqref{est1}, \eqref{prod}, \eqref{first} and \eqref{estiW}, one gets
\begin{align*}
\|\nabla \psi(t)\|_{L^2}&\leq \|\nabla \psi_0\|_{L^2}+C\|V_g\|_{L^\infty}\|V_0\|_{H^{3/2}}\int_0^t \|\nabla \psi(s)\|_{L^2}ds\\
& \quad +C(1+\alpha \|\psi_0\|_{L^2}^2)\int_0^t (1+\log(1+\sqrt{\alpha}\|\nabla \psi (s)\|_{L^2})) \|\nabla \psi(s)\|_{L^2}ds\\
&\leq \|\nabla \psi_0\|_{L^2}+C\delta\|\chi\|_{{\mathcal C}^1}\int_0^t \|\nabla \psi(s)\|_{L^2}ds\\
& \quad +C(1+\alpha \|\psi_0\|_{L^2}^2)\int_0^t (1+\log(1+\sqrt{\alpha_0}\|\nabla \psi (s)\|_{L^2})) \|\nabla \psi(s)\|_{L^2}ds
\end{align*}
where we used  $\|V_g\|_{L^\infty}\leq\delta$ and $\|V_0\|_{H^{3/2}}\leq C\|\chi\|_{{\mathcal C}^1}$. A logarithmic Gronwall lemma (see \cite{brezis-gallouet}) yields the \emph{a priori} estimate \eqref{apriori}. The proof of the proposition is complete.
\end{proof}

As application of this proposition, one deduces the following approximate controllability result for the nonlinear problem \eqref{schropois2}.
\begin{theorem}
For a generic triple 
$(Q(\Om),Q(\Gamma^D_g),\chi)$ in $\mathcal{P}$, defined as in \eqref{def-P}, for every $\psi_0\in H^1_0(Q(\Om),\C)$, $\psi_1\in L^2(Q(\Om),\C)$, with $\|\psi_0\|_{L^2}=\|\psi_1\|_{L^2}=1$, for every  tolerance $\varepsilon>0$, there exist a positive time $T$, a control
$V_g\in L^\infty([0,T],[0,\delta])$, and $\alpha_0>0$ such that, if
$0<\alpha\leq \alpha_0$, then the solution of  \eqref{schropois2} satisfies $\|\psi(T)-\psi_1\|_{L^2(Q(\Omega))}<\varepsilon$.
\end{theorem} 
\begin{proof}
Recall that, by Theorem \ref{grille-deformation}, for a generic triple 
$(Q(\Om),Q(\Gamma^D_g),\chi)$ in $\mathcal{P}$, the linear system \eqref{Qschro1} is approximately controllable. Fix then $Q$ and $\chi$ such that \eqref{Qschro1} is approximately controllable.
Fix $\psi_0\in H^1_0(Q(\Om),\C)$, $\psi_1\in L^2(Q(\Om),\C)$, with $\|\psi_0\|_{L^2}=\|\psi_1\|_{L^2}=1$ and $\varepsilon>0$. 
Then there  exist  $T>0$ and 
$V_g\in L^\infty([0,T],[0,\delta])$ such that
the solution $\psi_{\mathrm{lin}}$ of the linear equation \eqref{Qschro1} with initial condition $\psi_0$ corresponding to 
$V_g$ satisfies $\|\psi_{\mathrm{lin}}(T)-\psi_1\|_{L^2}<\varepsilon/2$.

Then, the solution $\psi(t,x)$ of the nonlinear equation \eqref{schropois2} with initial condition $\psi_0$ corresponding to the control $V_g$ reads
 \begin{align}
 \psi(t,\cdot)&=e^{-i t \Delta}\psi_0(\cdot)+\int_0^t e^{-i (t-s) \Delta}(V_g(s) V_0(\cdot)+W_\psi(s,\cdot))\psi(s,\cdot)ds\nonumber\\
 &=\psi_{\mathrm{lin}}(t,\cdot)+\int_0^t e^{-i (t-s) \Delta}W_\psi(s,\cdot)\psi(s,\cdot)ds.\label{mild2}
 \end{align}
The $L^\infty$ norm of $W_\psi$ can be estimated by using elliptic regularity for \eqref{nonlinear-s}, a Sobolev embedding and the bound \eqref{apriori} given in Proposition \ref{propCauchy}: for all $t\leq T$,
\begin{align*}
\|W_\psi(t)\|_{L^\infty}&\leq C\|W_\psi(t)\|_{H^2}\leq \alpha C\||\psi(t)|^2\|_{L^2}= \alpha C\|\psi(t)\|_{L^4}^2\\
&\leq \alpha C\|\psi(t)\|_{H^1}^2\leq \alpha C \exp(ce^{cT}).
\end{align*}
Fixing an upper bound $\alpha_{\max}>0$ for $\alpha$,  the constant $c>0$ can be chosen independent of $\alpha$ and  depending only on $Q$, $\chi$, $\psi_0$ and $\delta$, which are all fixed. Hence, inserting this estimate in \eqref{mild2} yields
\begin{align*}
\|\psi(T)-\psi_{\mathrm{lin}}(T)\|_{L^2}\leq \int_0^t \|W_\psi(s,\cdot)\psi(s,\cdot)\|_{L^2}\,ds&\leq \int_0^t \|W_\psi(s,\cdot)\|_{L^\infty}\|\psi(s,\cdot)\|_{L^2}\,ds\\
&\leq \alpha C T  \exp(ce^{cT}),
\end{align*}
where we used $\|\psi\|_{L^2}=1$ and that $e^{-i \tau \Delta}$ preserves the $L^2$-norm. Then it suffices to take 
$$\alpha_0=\min\left(\alpha_{\max},\frac{\varepsilon}{2CT  \exp(ce^{cT})}\right)$$ and the theorem is proved.
\end{proof}

\medskip
\noindent{\bf Acknowledgment.}
The second author was partially supported by the ANR project OPTIFORM.\\

\def\cprime{$'$}

%
%

\end{document}